\documentclass[11pt]{article}

\usepackage{amsmath, amsfonts, amssymb, amsthm,  graphicx, mathtools, enumerate}

\usepackage[blocks, affil-it]{authblk}

\usepackage[numbers, square]{natbib}
\usepackage[CJKbookmarks=true,
            bookmarksnumbered=true,
			bookmarksopen=true,
			colorlinks=true,
			citecolor=red,
			linkcolor=blue,
			anchorcolor=red,
			urlcolor=blue]{hyperref}
\usepackage[usenames]{color}

\usepackage[letterpaper, left=1.2truein, right=1.2truein, top = 1.2truein, bottom = 1.2truein]{geometry}
\usepackage[ruled, vlined, lined, commentsnumbered]{algorithm2e}

\usepackage{prettyref,soul}

\newtheorem{lemma}{Lemma}[section]
\newtheorem{proposition}{Proposition}[section]
\newtheorem{thm}{Theorem}[section]

\newtheorem{corollary}{Corollary}[section]
\newtheorem{remark}{Remark}[section]

\newrefformat{eq}{(\ref{#1})}
\newrefformat{chap}{Chapter~\ref{#1}}
\newrefformat{sec}{Section~\ref{#1}}
\newrefformat{algo}{Algorithm~\ref{#1}}
\newrefformat{fig}{Fig.~\ref{#1}}
\newrefformat{tab}{Table~\ref{#1}}
\newrefformat{rmk}{Remark~\ref{#1}}
\newrefformat{clm}{Claim~\ref{#1}}
\newrefformat{def}{Definition~\ref{#1}}
\newrefformat{cor}{Corollary~\ref{#1}}
\newrefformat{lmm}{Lemma~\ref{#1}}
\newrefformat{lemma}{Lemma~\ref{#1}}
\newrefformat{prop}{Proposition~\ref{#1}}
\newrefformat{app}{Appendix~\ref{#1}}
\newrefformat{ex}{Example~\ref{#1}}
\newrefformat{exer}{Exercise~\ref{#1}}
\newrefformat{soln}{Solution~\ref{#1}}
\newrefformat{cond}{Condition~\ref{#1}}

\def\text#1{\mbox{\rm #1}}

\title{Bernstein-von Mises Theorems for Functionals of Covariance Matrix
\thanks{The research of Chao Gao and Harrison H. Zhou is supported
in part by NSF Grant DMS-1209191.}
}
\author{Chao Gao}
\author{Harrison H.~Zhou}
\affil{
Yale University
}

\date{November 30, 2014}

\begin{document}
\maketitle

\begin{abstract}
We provide a general theoretical framework to derive Bernstein-von Mises theorems for matrix functionals. The conditions on functionals and priors are explicit and easy to check. Results are obtained for various functionals including  entries of  covariance matrix, entries of  precision matrix, quadratic forms, log-determinant, eigenvalues in the Bayesian Gaussian covariance/precision matrix estimation setting, as well as for Bayesian linear and quadratic discriminant analysis.
\smallskip

\end{abstract}


\section{Introduction}

The celebrated Bernstein-von Mises (BvM) theorem \citep{laplace, bernstein27, mises28, lecam00,vaart00} justifies Bayesian methods from a frequentist point of view. It bridges the gap between Bayesians and frequentists. Consider a parametric model $\big(P_{\theta}:\theta\in\Theta\big)$, and a prior distribution $\theta\sim\Pi$. Suppose we have i.i.d. observations $X^n=(X_1,...,X_n)$ from the product measure $P_{\theta^*}^n$. Under some weak assumptions,  Bernstein-von Mises theorem shows that the conditional distribution of
$$\sqrt{n}(\theta-\hat{\theta}) |X^n$$
is asymptotically $N(0,V^2)$ under the distribution $P_{\theta^*}^n$ with some centering $\hat{\theta}$ and covariance $V^2$ when $n\rightarrow\infty$. In a local asymptotic normal (LAN) family, the centering $\hat{\theta}$ can be taken as the maximum likelihood estimator (MLE) and $V^2$ as the inverse of the Fisher information matrix. An immediate consequence of the Bernstein-von Mises theorem is that the distributions
$$\sqrt{n}(\theta-\hat{\theta}) |X^n\quad\text{and}\quad \sqrt{n}(\hat{\theta}-\theta) |\theta=\theta^*$$
are asymptotically the same under the sampling distribution $P_{\theta^*}^n$. Note that the first one, known as the posterior, is of interest to Bayesians, and the second one is of interest to frequentists in the large sample theory. Applications of Bernstein-von Mises theorem include constructing confidence sets from Bayesian methods with frequentist coverage guarantees.

Despite the success of BvM results in the classical parametric setting, little is known about the high-dimensional case, where the unknown parameter is of increasing or even infinite dimensions. The pioneering works of \cite{cox93} and \cite{freedman99} (see also \cite{johnstone10}) showed that generally BvM may not be true in  non-classical cases. Despite the negative results, further works on some notions of nonparametric BvM provide some positive answers. See, for example, \cite{leahu11, castillo13a, castillo2013bernstein, ray14}. In this paper, we consider the question whether it is possible to have BvM results for matrix functionals, such as matrix entries and eigenvalues, when the dimension of the matrix $p$ grows with the sample size $n$. 

This paper provides some positive answers to this question. To be specific, we consider a multivariate Gaussian likelihood and put a prior on the covariance matrix. We prove that the posterior distribution has a BvM behavior for various matrix functionals including entries of covariance matrix, entries of  precision matrix, quadratic forms, log-determinant, and eigenvalues. All of these conclusions are obtained from a general theoretical framework we provide in Section \ref{sec:framework}, where we propose explicit easy-to-check conditions on both functionals and priors. We illustrate the theory by both conjugate and non-conjugate priors.
A slight extension of the general framework leads to BvM results for discriminant analysis.  Both linear discriminant analysis (LDA) and quadratic discriminant analysis (QDA) are considered.

This work is inspired by a growing interest in studying the BvM phenomena on a low-dimensional functional of the whole parameter. That is, the asymptotic distribution of
$$\sqrt{n}(f(\theta)-\hat{f})|X^n,$$
with $f$ being a map from $\Theta$ to $\mathbb{R}^d$, where $d$ does not grow with $n$. A special case is the semiparametric setting, where $\theta=(\mu,\eta)$ contains both a parametric part $\mu$ and a nonparametric part $\eta$. The functional $f$ takes the form of $f(\mu,\eta)=\mu$. The works in this field are pioneered by \cite{kim04} in a right-censoring model and \cite{shen02} for a general theory in the semiparametric setting. However, the conditions provided by \cite{shen02} for BvM to hold are hard to check when specific examples are considered. To the best of our knowledge, the first general framework for semiparametric BvM with conditions cleanly stated and easy to check is the beautiful work by \cite{castillo12}, in which the recent advancement in Bayes nonparametrics such as \cite{barron99} and \cite{ghosal00} are nicely absorbed. \cite{rivoirard12} proves BvM for linear functionals for which the distribution of $\sqrt{n}(f(\theta)-\hat{f})|X^n$ converges to a mixture of normal instead of a normal. At the point when this paper is drafted, the most updated theory is due to \cite{castillo13b}, which provides conditions for BvM to hold for general functionals.
The general framework we provide for matrix functional BvM is greatly inspired by the framework developed in \cite{castillo13b} for functionals in nonparametrics. However, the theory in this paper is different from theirs since we can take advantage of the structure in the Gaussian likelihood and avoid unnecessary expansion and approximation. Hence, in the covariance matrix functional case, our assumptions can be significantly weaker.
 
The paper is organized as follows. In Section \ref{sec:framework}, we state the general theoretical framework of our results. It is illustrated with two priors, one conjugate prior and one non-conjugate prior. Section \ref{sec:example} considers specific examples of matrix functionals and the associated BvM results. The extension to discriminant analysis is developed in Section \ref{sec:DA}. Finally, we devote Section \ref{sec:disc} to some discussions on the assumptions and possible generalizations. Most of the proofs are gathered in Section \ref{sec:proof}. 

\subsection{Notation}

Given a matrix $A$, we use $||A||$ to denote its spectral norm, and $||A||_F$ to denote its Frobenius norm. The norm $||\cdot||$, when applied to a vector, is understood to be the usual vector norm. Let $S^{p-1}$ be the unit sphere in $\mathbb{R}^p$. For any $a,b\in\mathbb{R}$, we use notation $a\vee b=\max(a,b)$ and $a\wedge b=\min(a,b)$. The probability $P_{\Sigma}$ stands for $N(0,\Sigma)$ and $P_{(\mu,\Omega)}$ is for $N(\mu,\Omega^{-1})$. In most cases, we use $\Sigma$  to denote the covariance matrix, and $\Omega$ to denote the precision matrix (including those with superscripts or subscripts). The notation $\mathbb{P}$ is for a generic probability, whenever the distribution is clear in the context. We use $O_P(\cdot)$ and $o_P(\cdot)$ to denote stochastic orders under the sampling distribution of the data. We use $C$ to indicate constants throughout the paper. They may be different from line to line.

\section{A General Framework} \label{sec:framework}

Consider i.i.d. samples $X^n=(X_1,...,X_n)$ drawn from $N(0,\Sigma^*)$, where $\Sigma^*$ is a $p\times p$ covariance matrix with inverse $\Omega^*$. A Bayes method puts a prior $\Pi$ on the precision matrix $\Omega$, and the posterior distribution is defined as
$$\Pi(B|X^n)=\frac{\int_B\exp\Big(l_n(\Omega)\Big)d\Pi(\Omega)}{\int\exp\Big(l_n(\Omega)\Big)d\Pi(\Omega)},$$
where $l_n(\Omega)$ is the log-likelihood of $N(0,\Omega^{-1})$ defined as
$$l_n(\Omega)=\frac{n}{2}\log\det(\Omega)-\frac{n}{2}\text{tr}(\Omega\hat{\Sigma}),\quad\text{where }\hat{\Sigma}=\frac{1}{n}\sum_{i=1}^nX_iX_i^T.$$
We deliberately omit the logarithmic normalizing constant in $l_n(\Omega)$ for simplicity and it will not affect the definition of the posterior distribution.
Note that specifying a prior on the precision matrix $\Omega$ is equivalent to specifying a prior on the covariance matrix $\Omega^{-1}$. The goal of this work is to show that the asymptotic distribution of the functional $f(\Omega)$ under the posterior distribution is approximately normal, i.e., 
$$\Pi\Big(\sqrt{n}V^{-1}\big(f(\Omega)-\hat{f}\big)\leq t|X^n\Big)\rightarrow \mathbb{P}(Z\leq t),$$
where 
$Z\sim N(0,1)$,
as $(n,p)\rightarrow\infty$ jointly with some appropriate centering $\hat{f}$ and variance $V^2$. In this paper, we choose the centering $\hat{f}$ to be the sample version of $f(\Omega)=f(\Sigma^{-1})$, where $\Sigma$ is replaced by the sample covariance $\hat{\Sigma}$, and compare the BvM results with the classical asymptotical normality for $\hat{f}$ in the frequentist sense. Other centering $\hat{f}$, including bias correction on the sample version, will be considered in the future work.

We first provide a framework for approximately linear functionals, and then use the general theory to derive results for specific examples of priors and functionals. For clarity of presentation, we consider the cases of functionals of $\Sigma$ and functionals of $\Omega$ separately. Though a functional of $\Sigma$ is also a functional of $\Omega$, we treat them separately, since some functional may be ``more linear" in $\Sigma$ than in $\Omega$, or the other way around.

\subsection{Functional of Covariance Matrix}

Let us first consider a functional of $\Sigma$, $f=\phi(\Sigma)$. The functional is approximately linear in a neighborhood of the truth. We assume there is a set $A_n$ satisfying
\begin{equation}
A_n\subset\left\{||\Sigma-\Sigma^*||\leq\delta_n\right\}, \label{eq:covsubset}
\end{equation}
for  any sequence $\delta_n=o(1)$, on which $\phi(\Sigma)$ is approximately linear in the sense that
there exists a symmetric matrix $\Phi$ such that
\begin{equation}
\sup_{A_n}\sqrt{n}\left\|\Sigma^{*1/2}\Phi\Sigma^{*1/2}\right\|_F^{-1}\left|\phi(\Sigma)-\phi(\hat{\Sigma})-\text{tr}\Big((\Sigma-\hat{\Sigma})\Phi\Big)\right|=o_P(1). \label{eq:approxlincov}
\end{equation}
The main result is stated in the following theorem.
\begin{thm} \label{thm:main1}
Under the assumptions of (\ref{eq:approxlincov}) and $||\Sigma^*||\vee||\Omega^*||= O(1)$, if for a given prior $\Pi$, the following two conditions are satisfied:
\begin{enumerate}
\item $\Pi(A_n|X^n)=1-o_P(1)$,
\item For any fixed $t\in\mathbb{R}$, $\frac{\int_{A_n}\exp\Big(l_n(\Omega_t)\Big)d\Pi(\Omega)}{\int_{A_n}\exp\Big(l_n(\Omega)\Big)d\Pi(\Omega)}=1+o_P(1)$ for the perturbed precision matrix
$$\Omega_t=\Omega+\frac{\sqrt{2}t}{\sqrt{n}\left\|\Sigma^{*1/2}\Phi\Sigma^{*1/2}\right\|_F}\Phi,$$
\end{enumerate}
then
$$\sup_{t\in\mathbb{R}}\left|\Pi\Bigg(\frac{\sqrt{n}\big(\phi(\Sigma)-\phi(\hat{\Sigma})\big)}{\sqrt{2}\left\|\Sigma^{*1/2}\Phi\Sigma^{*1/2}\right\|_F}\leq t\Big|X^n\Bigg)-\mathbb{P}\big(Z\leq t\big)\right|=o_P(1),$$
where $Z\sim N(0,1)$.
\end{thm}

The theorem gives explicit conditions on both prior and functional. The first condition says that the posterior distribution concentrates on a neighborhood of the truth under the spectral norm, on which the functional is approximately linear. The second condition says that the bias caused by the shifted parameter can be absorbed by the posterior distribution. Under both conditions,
Theorem \ref{thm:main1} shows that the asymptotic posterior distribution of $\phi(\Sigma)$ is
$$N\left(\phi(\hat{\Sigma}),2n^{-1}\left\|{\Sigma^*}^{1/2}\Phi{\Sigma^*}^{1/2}\right\|_F^2\right).$$

\subsection{Functional of Precision Matrix}

We state a corresponding theorem for functionals of precision matrix in this section. The condition for linear approximation is slightly different. Consider the functional $f=\psi(\Omega)$. Let $A_n$ be a set satisfying
\begin{equation}
A_n\subset\left\{\sqrt{rp}||\Sigma-\Sigma^*||\leq \delta_n\right\}, \label{eq:precisionsubset}
\end{equation}
for some integer $r>0$ and any sequence $\delta_n=o(1)$. We assume the functional $\psi(\Omega)$ is approximately linear on $A_n$ in the sense that
 there exists a symmetric matrix $\Psi$ satisfying $\text{rank}(\Psi)\leq r$, such that
\begin{equation}
\sup_{A_n}\sqrt{n}\left\|\Omega^{*1/2}\Psi\Omega^{*1/2}\right\|_F^{-1}\left|\psi(\Omega)-\psi(\hat{\Sigma}^{-1})-\text{tr}\Big((\Omega-\hat{\Sigma}^{-1})\Psi\Big)\right|=o_P(1). \label{eq:approxlinprecision}
\end{equation}
The main result is stated in the following theorem.
\begin{thm} \label{thm:main2}
Under the assumptions of (\ref{eq:approxlinprecision}), $rp^2/n=o(1)$ and  $||\Sigma^*||\vee||\Omega^*||= O(1)$, if for a given prior $\Pi$, the following conditions are satisfied:
\begin{enumerate}
\item $\Pi(A_n|X^n)=1-o_P(1)$,
\item For any fixed $t\in\mathbb{R}$, $\frac{\int_{A_n}\exp\Big(l_n(\Omega_t)\Big)d\Pi(\Omega)}{\int_{A_n}\exp\Big(l_n(\Omega)\Big)d\Pi(\Omega)}=1+o_P(1)$ for the perturbed precision matrix
$$\Omega_t=\Omega-\frac{\sqrt{2}t}{\sqrt{n}\left\|\Omega^{*1/2}\Psi\Omega^{*1/2}\right\|_F}\Omega^*\Psi\Omega^*,$$
\end{enumerate}
then
$$\sup_{t\in\mathbb{R}}\left|\Pi\Bigg(\frac{\sqrt{n}\big(\psi(\Omega)-\psi(\hat{\Sigma}^{-1}))\big)}{\sqrt{2}\left\|\Omega^{*1/2}\Psi\Omega^{*1/2}\right\|_F}\leq t\Big|X^n\Bigg)-\mathbb{P}\big(Z\leq t\big)\right|=o_P(1),$$
where $Z\sim N(0,1)$.
\end{thm}

\begin{remark}
The extra condition $rp^2/n=o(1)$ does not appear in Theorem \ref{thm:main1}. We show that this condition is indeed sharp for Theorem \ref{thm:main2} in Section \ref{sec:sharpthm2} in comparison with the asymptotics of MLE.
\end{remark}

\subsection{Priors}

In this section, we provide examples of priors. In particular, we consider both a conjugate prior and a non-conjugate prior. Note that the result of a conjugate prior can be derived by directly exploring the posterior form without applying our general theory. However, the general framework provided in this paper can handle both conjugate and non-conjugate priors in a unified way.

\subsubsection{Wishart Prior}

Consider the Wishart prior $\mathcal{W}_p(I,p+b-1)$ on $\Omega$ with density function
\begin{equation}
\frac{d\Pi(\Omega)}{d\Omega}\propto \exp\Bigg(\frac{b-2}{2}\log\det(\Omega)-\frac{1}{2}\text{tr}(\Omega)\Bigg), \label{eq:wishartdensity}
\end{equation}
supported on the set of symmetric positive semi-definite matrices.

\begin{lemma}\label{lem:wishart}
Assume $||\Sigma^*||\vee||\Omega^*||= O(1)$ and $p/n=o(1)$. Then, for any integer $b=O(1)$, the prior $\Pi=\mathcal{W}_p(I,p+b-1)$ satisfies the two conditions in Theorem \ref{thm:main1} for some $A_n$. If the extra assumption $rp^2/n=o(1)$ is made,  the two conditions in Theorem \ref{thm:main2} are also satisfied for some $A_n$.
\end{lemma}

\begin{remark}
In the proof of Lemma \ref{lem:wishart} (Section \ref{sec:proofwishart}), we set
$$A_n=\left\{||\Sigma-\Sigma^*||\leq M\sqrt{\frac{p}{n}}\right\},$$ for some $M>0$.
\end{remark}

\subsubsection{Gaussian Prior} \label{sec:gaussprior}

Consider Gaussian prior on $\Omega$ with density function
\begin{equation}
\frac{d\Pi(\Omega)}{d\Omega}\propto \exp\Big(-\frac{1}{2}||\Omega||_F^2\Big),\label{eq:gaussiandensity}
\end{equation}
supported on the following set
$$\left\{\Omega=\Omega^T, ||\Omega||<2\Lambda,||\Sigma||\leq 2\Lambda\right\},$$
for some constant $\Lambda>0$.

\begin{lemma}\label{lem:gaussprior}
Assume $||\Sigma^*||\vee||\Omega^*||\leq\Lambda=O(1)$ and $\frac{p^2\log n}{n}=o(1)$. The Gaussian prior $\Pi$ defined above satisfies the two conditions in Theorem \ref{thm:main1} for some appropriate $A_n$. If the extra assumption $\frac{rp^3\log n}{n}=o(1)$ is made, the two conditions in Theorem \ref{thm:main2} are also satisfied for some appropriate $A_n$.
\end{lemma}

\begin{remark}
In the proof of Lemma \ref{lem:gaussprior} (Section \ref{sec:proofgauss}), we set $$A_n=\left\{||\Sigma-\Sigma^*||_F\leq M\sqrt{\frac{p^2\log n}{n}}\right\},$$ for some constant $M>0$.
\end{remark}

\section{Examples of Matrix Functionals} \label{sec:example}

We consider various examples of functionals in this section. The two conditions of Theorem \ref{thm:main1} and Theorem \ref{thm:main2} are satisfied by Wishart prior and Gaussian prior, as is shown in Lemma \ref{lem:wishart} and Lemma \ref{lem:gaussprior} respectively. Hence, it is sufficient to check the approximate linearity of the functional with respect to $\Sigma$ or $\Omega$ for the BvM result to hold. Among the four examples we consider, the first two are exactly linear and the last two are approximately linear. In the below examples, $Z$ is always a random variable distributed as $N(0,1)$.

\subsection{Entry-wise Functional}

We consider the elementwise functional $\sigma_{ij}=\phi_{ij}(\Sigma)$ and $\omega_{ij}=\psi_{ij}(\Omega)$. Note that these two functionals are linear with respect to $\Sigma$ and $\Omega$ respectively. For $\sigma_{ij}$, we write
$$\sigma_{ij}=\text{tr}\Big(\Sigma\big(\frac{1}{2}E_{ij}+\frac{1}{2}E_{ji}\big)\Big),$$
where the matrix $E_{ij}$ is the $(i,j)$-th basis in $\mathbb{R}^{p\times p}$ with $1$ on its $(i,j)$-the element and $0$ elsewhere.  For $\omega_{ij}$, we write
$$\omega_{ij}=\text{tr}\Big(\Omega\big(\frac{1}{2}E_{ij}+\frac{1}{2}E_{ji}\big)\Big).$$
Note that $\text{rank}\Big(\frac{1}{2}E_{ij}+\frac{1}{2}E_{ji}\Big)\leq 2$. Hence, the corresponding matrices $\Phi$ and $\Psi$ in the linear expansion of $\phi$ and $\psi$ are $\frac{1}{2}E_{ij}+\frac{1}{2}E_{ji}$. In view of Theorem \ref{thm:main1} and Theorem \ref{thm:main2}, the asymptotic variance for $\sqrt{n}\big(\phi(\Sigma)-\phi(\hat{\Sigma})\big)$ is
$$2\left\|\Sigma^{*1/2}\Phi\Sigma^{*1/2}\right\|_F^2=\sigma_{ii}^*\sigma_{jj}^*+\sigma_{ij}^{*2}.$$
The asymptotic variance for $\sqrt{n}\big(\psi(\Omega)-\psi(\hat{\Sigma}^{-1})\big)$ is
$$2\left\|\Omega^{*1/2}\Psi\Omega^{*1/2}\right\|_F^2=\omega_{ii}^*\omega_{jj}^*+\omega_{ij}^{*2}.$$
Plugging these quantities in Theorem \ref{thm:main1}, Theorem \ref{thm:main2}, Lemma \ref{lem:wishart}, and Lemma \ref{lem:gaussprior}, we have the following  Bernstein-von Mises results.

\begin{corollary} \label{cor:wishartbvmelement}
Consider the Wishart prior $\Pi=\mathcal{W}_p(I,p+b-1)$ in (\ref{eq:wishartdensity}) with integer $b=O(1)$. Assume $||\Sigma^*||\vee||\Omega^*||= O(1)$ and $p/n=o(1)$,  then we have
$$P_{\Sigma^*}^n\sup_{t\in\mathbb{R}}\left|\Pi\Bigg(\frac{\sqrt{n}(\sigma_{ij}-\hat{\sigma}_{ij})}{\sqrt{\sigma_{ii}^*\sigma_{jj}^*+\sigma_{ij}^{*2}}}\leq t\Big|X^n\Bigg)-\mathbb{P}\big(Z\leq t\big)\right|\rightarrow 0,$$
where $\hat{\sigma}_{ij}$ is the $(i,j)$-th element of the sample covariance $\hat{\Sigma}$. If we additionally assume $p^2/n=o(1)$, then
$$P_{\Sigma^*}^n\sup_{t\in\mathbb{R}}\left|\Pi\Bigg(\frac{\sqrt{n}(\omega_{ij}-\hat{\omega}_{ij})}{\sqrt{\omega_{ii}^*\omega_{jj}^*+\omega_{ij}^{*2}}}\leq t\Big|X^n\Bigg)-\mathbb{P}\big(Z\leq t\big)\right|\rightarrow 0,$$
where $\hat{\omega}_{ij}$ is the $(i,j)$-th element of $\hat{\Sigma}^{-1}$.
\end{corollary}

\begin{corollary}
Consider the Gaussian prior $\Pi$ in (\ref{eq:gaussiandensity}). Assume $||\Sigma^*||\vee||\Omega^*||\leq\Lambda= O(1)$ and $\frac{p^2\log n}{n}=o(1)$, then we have
$$P_{\Sigma^*}^n\sup_{t\in\mathbb{R}}\left|\Pi\Bigg(\frac{\sqrt{n}(\sigma_{ij}-\hat{\sigma}_{ij})}{\sqrt{\sigma_{ii}^*\sigma_{jj}^*+\sigma_{ij}^{*2}}}\leq t\Big|X^n\Bigg)-\mathbb{P}\big(Z\leq t\big)\right|\rightarrow 0.$$
If we additionally assume $\frac{p^3\log n}{n}=o(1)$, then
$$P_{\Sigma^*}^n\sup_{t\in\mathbb{R}}\left|\Pi\Bigg(\frac{\sqrt{n}(\omega_{ij}-\hat{\omega}_{ij})}{\sqrt{\omega_{ii}^*\omega_{jj}^*+\omega_{ij}^{*2}}}\leq t\Big|X^n\Bigg)-\mathbb{P}\big(Z\leq t\big)\right|\rightarrow 0,$$
where $\hat{\sigma}_{ij}$ and $\hat{\omega}_{ij}$ are defined in Corollary \ref{cor:wishartbvmelement}.
\end{corollary}

\subsection{Quadratic Form}

Consider the functional $\phi_v(\Sigma)=v^T\Sigma v=\text{tr}(\Sigma vv^T)$ and $\psi_v(\Omega)=v\Omega v^T=\text{tr}(\Omega vv^T)$ for some $v\in \mathbb{R}^p$. Therefore, the corresponding matrices $\Phi$ and $\Psi$ are $vv^T$. It is easy to see that $\text{rank}(vv^T)=1$. The asymptotic variances are
$$2\left\|\Sigma^{*1/2}\Phi\Sigma^{*1/2}\right\|_F^2=2|v^T\Sigma^*v|^2,\quad 2\left\|\Omega^{*1/2}\Psi\Omega^{*1/2}\right\|_F^2=2|v^T\Omega^*v|^2.$$
Plugging these representations in Theorem \ref{thm:main1}, Theorem \ref{thm:main2}, Lemma \ref{lem:wishart} and Lemma \ref{lem:gaussprior}, we have the following  Bernstein-von Mises results.

\begin{corollary} \label{cor:wishartbvmquadratic}
Consider the Wishart prior $\Pi=\mathcal{W}_p(I,p+b-1)$ in (\ref{eq:wishartdensity}) with integer $b=O(1)$. Assume $||\Sigma^*||\vee||\Omega^*||= O(1)$ and $p/n=o(1)$, then we have
$$P_{\Sigma^*}^n\sup_{t\in\mathbb{R}}\left|\Pi\Bigg(\frac{\sqrt{n}(v^T\Sigma v-v^T\hat{\Sigma} v)}{\sqrt{2}|v^T\Sigma^* v|}\leq t\Big|X^n\Bigg)-\mathbb{P}\big(Z\leq t\big)\right|\rightarrow 0.$$
If we additionally assume $p^2/n=o(1)$, then
$$P_{\Sigma^*}^n\sup_{t\in\mathbb{R}}\left|\Pi\Bigg(\frac{\sqrt{n}(v^T\Omega v-v^T\hat{\Sigma}^{-1}v)}{\sqrt{2}|v^T\Omega^* v|}\leq t\Big|X^n\Bigg)-\mathbb{P}\big(Z\leq t\big)\right|\rightarrow 0.$$
\end{corollary}

\begin{corollary} \label{cor:gaussbvmquadratic}
Consider the Gaussian prior $\Pi$ in (\ref{eq:gaussiandensity}). Assume $||\Sigma^*||\vee||\Omega^*||\leq\Lambda= O(1)$ and $\frac{p^2\log n}{n}=o(1)$, then we have
$$P_{\Sigma^*}^n\sup_{t\in\mathbb{R}}\left|\Pi\Bigg(\frac{\sqrt{n}(v^T\Sigma v-v^T\hat{\Sigma} v)}{\sqrt{2}|v^T\Sigma^* v|}\leq t\Big|X^n\Bigg)-\mathbb{P}\big(Z\leq t\big)\right|\rightarrow 0.$$
If we additionally assume $\frac{p^3\log n}{n}=o(1)$, then
$$P_{\Sigma^*}^n\sup_{t\in\mathbb{R}}\left|\Pi\Bigg(\frac{\sqrt{n}(v^T\Omega v-v^T\hat{\Sigma}^{-1}v)}{\sqrt{2}|v^T\Omega^* v|}\leq t\Big|X^n\Bigg)-\mathbb{P}\big(Z\leq t\big)\right|\rightarrow 0.$$
\end{corollary}

\begin{remark}
The entry-wise functional and the quadratic form are both special cases of
the functional $u^T\Sigma v$ for some $u,v\in\mathbb{R}^p$. It is direct to apply the general framework to this functional and obtain the result
$$P_{\Sigma^*}^n\sup_{t\in\mathbb{R}}\left|\Pi\Bigg(\frac{\sqrt{n}(u^T\Sigma v-u^T\hat{\Sigma} v)}{\sqrt{|u^T\Sigma^*v|^2+|u^T\Sigma^*u||v^T\Sigma^*v|}}\leq t\Big|X^n\Bigg)-\mathbb{P}\big(Z\leq t\big)\right|\rightarrow 0.$$
Similarly, for the functional $u^T\Omega v$ for some  $u,v\in\mathbb{R}^p$, we have
$$P_{\Sigma^*}^n\sup_{t\in\mathbb{R}}\left|\Pi\Bigg(\frac{\sqrt{n}(u^T\Omega v-u^T\hat{\Sigma}^{-1} v)}{\sqrt{|u^T\Omega^*v|^2+|u^T\Omega^*u||v^T\Omega^*v|}}\leq t\Big|X^n\Bigg)-\mathbb{P}\big(Z\leq t\big)\right|\rightarrow 0,$$
Both results can be derived
under the same conditions of Corollary \ref{cor:wishartbvmquadratic} and Corollary \ref{cor:gaussbvmquadratic}.
\end{remark}

\subsection{Log Determinant}

In this section, we consider the log-determinant functional. That is $\phi(\Sigma)=\log\det(\Sigma)$. Different from entry-wise functional and quadratic form, we do not need to consider $\log\det(\Omega)$ because of the simple observation
$$\log\det(\Omega)=-\log\det(\Sigma).$$
The following lemma establishes the approximate linearity of $\log\det(\Sigma)$.
\begin{lemma} \label{lem:logdet}
Assume $||\Sigma^*||\vee||\Omega^*||= O(1)$ and $p^3/n=o(1)$, then for any $\delta_n=o(1)$, we have
$$\sup_{\{\sqrt{n/p}||\Sigma-\Sigma^*||_F^2\vee \sqrt{p}||\Sigma-\Sigma^*||_F\leq\delta_n\}}\sqrt{\frac{n}{p}}\left|\log\det(\Sigma)-\log\det(\hat{\Sigma})-\text{tr}\Big((\Sigma-\hat{\Sigma})\Omega^*\Big)\right|=o_P(1).$$
\end{lemma}
By Lemma \ref{lem:logdet}, the corresponding matrix $\Phi$ is $\Omega^*$. The asymptotic variance of $\sqrt{n}\Big(\phi(\Sigma)-\phi(\hat{\Sigma})\Big)$ is
$$2\left\|\Sigma^{*1/2}\Phi\Sigma^{*1/2}\right\|_F^2=2p.$$

\begin{corollary}
Consider the Wishart prior $\Pi=\mathcal{W}_p(I,p+b-1)$ in (\ref{eq:wishartdensity}) with integer $b=O(1)$. Assume $||\Sigma^*||\vee||\Omega^*||= O(1)$ and $p^3/n=o(1)$, then we have
$$P_{\Sigma^*}^n\sup_{t\in\mathbb{R}}\left|\Pi\Bigg(\sqrt{\frac{n}{2p}}\big(\log\det(\Sigma)-\log\det(\hat{\Sigma})\big)\leq t\Big|X^n\Bigg)-\mathbb{P}\Big(Z\leq t\Big)\right|\rightarrow 0,$$
where $\hat{\Sigma}$ is the sample covariance matrix.
\end{corollary}
\begin{proof}
By Theorem \ref{thm:main1} and Lemma \ref{lem:wishart}, we only need to check the approximate linearity of the functional. According to the proof of Lemma \ref{lem:wishart}, the choice of $A_n$ such that $\Pi(A_n|X^n)=1-o_P(1)$ is
$$A_n=\left\{||\Sigma-\Sigma^*||\leq M\sqrt{\frac{p}{n}}\right\},$$
for some $M>0$. This implies $||\Sigma-\Sigma^*||_F\leq M\sqrt{\frac{p^2}{n}}$. Therefore,
$$A_n\subset \{\sqrt{n/p}||\Sigma-\Sigma^*||_F^2\vee \sqrt{p}||\Sigma-\Sigma^*||_F\leq\delta_n\},$$
for some $\delta_n=o(1)$. By Lemma \ref{lem:logdet}, we have
$$\sup_{A_n}\sqrt{\frac{n}{p}}\left|\log\det(\Sigma)-\log\det(\hat{\Sigma})-\text{tr}\Big((\Sigma-\hat{\Sigma})\Omega^*\Big)\right|=o_P(1),$$
and the approximate linearity holds.
\end{proof}

\begin{corollary}
Consider the Gaussian prior $\Pi$ in (\ref{eq:gaussiandensity}). Assume $||\Sigma^*||\vee||\Omega^*||\leq\Lambda= O(1)$ and $\frac{p^3(\log n)^2}{n}=o(1)$, then we have
$$P_{\Sigma^*}^n\sup_{t\in\mathbb{R}}\left|\Pi\Bigg(\sqrt{\frac{n}{2p}}\big(\log\det(\Sigma)-\log\det(\hat{\Sigma})\big)\leq t\Big|X^n\Bigg)-\mathbb{P}\Big(Z\leq t\Big)\right|\rightarrow 0,$$
where $\hat{\Sigma}$ is the sample covariance matrix.
\end{corollary}
\begin{proof}
The proof of this corollary is the same as the proof of the last one using Wishart prior. The only difference is that the choice of $A_n$, according to the proof of Lemma \ref{lem:gaussprior}, is
$$A_n=\left\{||\Sigma-\Sigma^*||_F\leq M\sqrt{\frac{p^2\log n}{n}}\right\},$$
for some $M>0$. Therefore,
$$A_n\subset \{\sqrt{n/p}||\Sigma-\Sigma^*||_F^2\vee \sqrt{p}||\Sigma-\Sigma^*||_F\leq\delta_n\},$$
for some $\delta_n=o(1)$ under the assumption, and the approximate linearity holds.
\end{proof}

One immediate consequence of the result is the Bernstein-von Mises result for the entropy functional, defined as
$$H(\Sigma)=\frac{p}{2}+\frac{p\log(2\pi)}{2}+\frac{\log\det(\Sigma)}{2}.$$
Then it is direct that
$$\sqrt{\frac{2n}{p}}\Big(H(\Sigma)-H(\hat{\Sigma})\Big)\Big| X^n \approx N(0,1).$$

\subsection{Eigenvalues}

In this section, we consider the eigenvalue functional. In particular, let $\{\lambda_{m}(\Sigma)\}_{m=1}^p$ be eigenvalues of the matrix $\Sigma$ with decreasing order. We investigate the posterior distribution of $\lambda_m(\Sigma)$ for each $m=1,...,p$. Define the eigen-gap
$$\delta=\begin{cases}
|\lambda_1(\Sigma^*)-\lambda_2(\Sigma^*)| & m=1,\\
\min\{|\lambda_m(\Sigma^*)-\lambda_{m-1}(\Sigma^*)|,|\lambda_m(\Sigma^*)-\lambda_{m+1}(\Sigma^*)|\} & m=2,3,...,p-1, \\
|\lambda_{m-1}(\Sigma^*)-\lambda_m(\Sigma^*)| & m=p.
\end{cases}$$
The asymptotic order of $\delta$ plays an important role in the theory. The following lemma characterizes the approximate linearity of $\lambda_m(\Sigma)$.

\begin{lemma} \label{lem:coveigen}
Assume $||\Sigma^*||\vee||\Omega^*||= O(1)$ and $\frac{p}{\delta\sqrt{n}}=o(1)$, then for any $\delta_n=o(1)$, we have
$$\sup_{\{\delta^{-1}\sqrt{n}||\Sigma-\Sigma^*||^2\vee (\delta^{-1}+\sqrt{p})||\Sigma-\Sigma^*||\leq\delta_n\}}\sqrt{n}\left|\lambda_m(\Sigma^*)\right|^{-1}\left|\lambda_m(\Sigma)-\lambda_m(\hat{\Sigma})-\text{tr}\Big((\Sigma-\hat{\Sigma})u_m^*u_m^{*T}\Big)\right|=o_P(1),$$
where $u_m^*$ is the $m$-th eigenvector of $\Sigma^*$.
\end{lemma}
Lemma \ref{lem:coveigen} implies that the corresponding $\Phi$ in the linear expansion of $\phi(\Sigma)$ is $u_m^*u_m^{*T}$, and the asymptotic variance is
$$2\left\|\Sigma^{*1/2}\Phi\Sigma^{*1/2}\right\|_F^2=2|\lambda_m(\Sigma^*)|^2.$$

We also consider eigenvalues of the precision matrix. With slight abuse of notation, we define the eigengap of $\lambda_m(\Omega^*)$ to be
$$\delta=\begin{cases}
|\lambda_1(\Omega^*)-\lambda_2(\Omega^*)| & m=1,\\
\min\{|\lambda_m(\Omega^*)-\lambda_{m-1}(\Omega^*)|,|\lambda_m(\Omega^*)-\lambda_{m+1}(\Omega^*)|\} & m=2,3,...,p-1, \\
|\lambda_{m-1}(\Omega^*)-\lambda_m(\Omega^*)| & m=p.
\end{cases}$$

The approximate linearity of $\lambda_m(\Omega)$ is established in the following lemma.

\begin{lemma} \label{lem:precisioneigen}
Assume $||\Sigma^*||\vee||\Omega^*||= O(1)$, then for any $\delta_n=o(1)$, we have
$$\sup_{\{\delta^{-1}\sqrt{n}||\Sigma-\Sigma^*||^2\vee(\delta^{-1}+\sqrt{p})||\Sigma-\Sigma^*||\leq\delta_n\}} \sqrt{n}|\lambda_m(\Omega^*)|^{-1}\left|\lambda_m(\Omega)-\lambda_m(\hat{\Sigma}^{-1})-\text{tr}\Big((\Omega-\hat{\Sigma}^{-1})u_m^*u_m^{*T}\Big)\right|=o(1),$$
where $u_m^*$ is the $m$-th eigenvector of $\Omega^*$.
\end{lemma}
Similarly, Lemma \ref{lem:precisioneigen} implies that the corresponding $\Psi$ in the linear expansion of $\psi(\Omega)$ is $u_m^*u_m^{*T}$, and the asymptotic variance is
$$2\left\|\Omega^{*1/2}\Psi\Omega^{*1/2}\right\|_F^2=2|\lambda_m(\Omega^*)|^2.$$

Plugging the above lemmas into our general framework, we get the following corollaries.

\begin{corollary}
Consider the Wishart prior $\Pi=\mathcal{W}_p(I,p+b-1)$ in (\ref{eq:wishartdensity}) with integer $b=O(1)$. Assume $||\Sigma^*||\vee||\Omega^*||= O(1)$ and $\frac{p}{\delta\sqrt{n}}=o(1)$,  then we have
$$P_{\Sigma^*}^n\sup_{t\in\mathbb{R}}\left|\Pi\Bigg(\sqrt{\frac{n}{\sqrt{2}\lambda_m(\Sigma^*)}}\big(\lambda_m(\Sigma)-\lambda_m(\hat{\Sigma})\big)\leq t|X^n\Bigg)-\mathbb{P}\Big(Z\leq t\Big)\right|\rightarrow 0,$$
where $\hat{\Sigma}$ is the sample covariance matrix. If we instead assume $\frac{p}{\delta\sqrt{n}}=o(1)$ with $\delta$ being the eigengap of $\lambda_m(\Omega^*)$, then
$$P_{\Sigma^*}^n\sup_{t\in\mathbb{R}}\left|\Pi\Big(\sqrt{\frac{n}{\sqrt{2}\lambda_m(\Omega^*)}}\big(\lambda_m(\Omega)-\lambda_m(\hat{\Sigma}^{-1})\big)\leq t|X^n\Big)-\mathbb{P}\Big(Z\leq t\Big)\right|\rightarrow 0.$$
\end{corollary}
\begin{proof}
We only need to check the approximate linearity. According to Lemma \ref{lem:wishart}, the choice of $A_n$ is
$$A_n=\left\{||\Sigma-\Sigma^*||\leq M\sqrt{\frac{p}{n}}\right\},$$
for some $M>0$. The assumption $\frac{p}{\delta\sqrt{n}}=o(1)$ implies
$$\delta^{-1}\sqrt{n}||\Sigma-\Sigma^*||^2\vee (\delta^{-1}+\sqrt{p})||\Sigma-\Sigma^*||= o(1),$$ on the set $A_n$. By Lemma \ref{lem:coveigen} and Lemma \ref{lem:precisioneigen}, we have
$$\sup_{A_n}\sqrt{n}|\lambda_m(\Sigma^*)|^{-1}\left|\lambda_m(\Sigma)-\lambda_m(\hat{\Sigma})-\text{tr}\Big((\Sigma-\hat{\Sigma})u_m^*u_m^{*T}\Big)\right|=o_P(1),$$
and
$$\sup_{A_n} \sqrt{n}|\lambda_m(\Omega^*)|^{-1}\left|\lambda_m(\Omega)-\lambda_m(\hat{\Sigma}^{-1})-\text{tr}\Big((\Omega-\hat{\Sigma}^{-1})u_m^*u_m^{*T}\Big)\right|=o_P(1).$$
\end{proof}

\begin{corollary}
Consider the Gaussian prior $\Pi$ in (\ref{eq:gaussiandensity}). Assume $||\Sigma^*||\vee||\Omega^*||\leq\Lambda=O(1)$ and $\frac{p^2\log n}{\delta\sqrt{n}}=o(1)$, then we have
$$P_{\Sigma^*}^n\sup_{t\in\mathbb{R}}\left|\Pi\Bigg(\sqrt{\frac{n}{\sqrt{2}\lambda_m(\Sigma^*)}}\big(\lambda_m(\Sigma)-\lambda_m(\hat{\Sigma})\big)\leq t|X^n\Bigg)-\mathbb{P}\Big(Z\leq t\Big)\right|\rightarrow 0,$$
where $\hat{\Sigma}$ is the sample covariance matrix. If we instead assume $\frac{p^2\log n}{\delta\sqrt{n}}=o(1)$ with $\delta$ being the eigengap of $\lambda_m(\Omega^*)$, then 
$$P_{\Sigma^*}^n\sup_{t\in\mathbb{R}}\left|\Pi\Big(\sqrt{\frac{n}{\sqrt{2}\lambda_m(\Omega^*)}}\big(\lambda_m(\Omega)-\lambda_m(\hat{\Sigma}^{-1})\big)\leq t|X^n\Big)-\mathbb{P}\Big(Z\leq t\Big)\right|\rightarrow 0.$$
\end{corollary}
\begin{proof}
We only need to check the approximate linearity. According to Lemma \ref{lem:gaussprior}, the choice of $A_n$ is
$$A_n=\left\{||\Sigma-\Sigma^*||_F\leq M\sqrt{\frac{p^2\log n}{n}}\right\},$$
for some $M>0$. The assumption $\frac{p^2\log n}{\delta\sqrt{n}}=o(1)$ implies
$$\delta^{-1}\sqrt{n}||\Sigma-\Sigma^*||^2\vee (\delta^{-1}+\sqrt{p})||\Sigma-\Sigma^*||= o(1),$$ on the set $A_n$. By Lemma \ref{lem:coveigen} and Lemma \ref{lem:precisioneigen}, we have
$$\sup_{A_n}\sqrt{n}|\lambda_m(\Sigma^*)|^{-1}\left|\lambda_m(\Sigma)-\lambda_m(\hat{\Sigma})-\text{tr}\Big((\Sigma-\hat{\Sigma})u_m^*u_m^{*T}\Big)\right|=o_P(1),$$
and
$$\sup_{A_n} \sqrt{n}|\lambda_m(\Omega^*)|^{-1}\left|\lambda_m(\Omega)-\lambda_m(\hat{\Sigma}^{-1})-\text{tr}\Big((\Omega-\hat{\Sigma}^{-1})u_m^*u_m^{*T}\Big)\right|=o_P(1).$$
\end{proof}

\section{Discriminant Analysis} \label{sec:DA}

In this section, we generalize the theory in Section \ref{sec:framework} to handle the BvM theorem in discriminant analysis. Let $X^n=(X_1,...,X_n)$ and $Y^n=(Y_1,...,Y_n)$ be $n$ i.i.d. training samples, where
$$X_i\sim N(\mu_X^*,\Omega_X^{*-1}),\quad Y_i\sim N(\mu_Y^*,\Omega_Y^{*-1}).$$
The discriminant analysis problem is to predict whether an independent new sample $z$ is from the $X$-class or $Y$-class. For a given $(\mu_X,\mu_Y,\Omega_X,\Omega_Y)$, Fisher's QDA rule can be written as
$$\Delta(\mu_X,\mu_Y,\Omega_X,\Omega_Y)=-(z-\mu_X)^T\Omega_X(z-\mu_X)+(z-\mu_Y)^T\Omega_Y(z-\mu_Y)+\log\frac{\det(\Omega_X)}{\det(\Omega_Y)}.$$
In this section, we are going to find the asymptotic posterior distribution
$$\sqrt{n}V^{-1}\Big(\Delta(\mu_X,\mu_Y,\Omega_X,\Omega_Y)-\hat{\Delta}\Big)\Big| X^n,Y^n,z,$$
with some appropriate variance $V^2$ and some prior distribution. Since the result is conditional on the new observation $z$, we treat it as a fixed (non-random) vector in this section without loss of generality.

Note that when $\Omega_X=\Omega_Y$ is assumed, the QDA rule can be reduced to the LDA rule. We give general results for Bernstein-von Mises theorem to hold in both cases respectively.

\subsection{Linear Discriminant Analysis}

Assume $\Omega_X^*=\Omega_Y^*$. For a given prior $\Pi$, the posterior distribution for LDA is defined as
$$\Pi(B|X^n,Y^n)=\frac{\int_B\exp\Big(l_n(\mu_X,\mu_Y,\Omega)\Big)d\Pi(\mu_X,\mu_Y,\Omega)}{\int\exp\Big(l_n(\mu_X,\mu_Y,\Omega)\Big)d\Pi(\mu_X,\mu_Y,\Omega)},$$
where $l_n(\mu_X,\mu_Y,\Omega)$ is the log-likelihood function decomposed as
$$l_n(\mu_X,\mu_Y,\Omega)=l_X(\mu_X,\Omega)+l_Y(\mu_Y,\Omega),$$
where
$$l_X(\mu_X,\Omega)=\frac{n}{2}\log\det(\Omega)-\frac{n}{2}\text{tr}(\Omega\tilde{\Sigma}_X),$$
with $\tilde{\Sigma}_X=\frac{1}{n}\sum_{i=1}^n(X_i-\mu_X)(X_i-\mu_X)^T$, and $l_Y(\mu_Y,\Omega)$ is defined in the similar way.

Consider the LDA functional
$$\Delta(\mu_X,\mu_Y,\Omega)=-(z-\mu_X)^T\Omega(z-\mu_X)+(z-\mu_Y)^T\Omega(z-\mu_Y).$$
Define the following quantities
$$\Phi=\frac{1}{2}\Omega^*\Big((z-\mu_X^*)(z-\mu_X^*)^T-(z-\mu_Y^*)(z-\mu_Y^*)^T\Big)\Omega^*,$$
$$\xi_X=2(z-\mu_X^*),\quad \xi_Y=2(\mu_Y^*-z),$$
$$\Omega_t=\Omega+\frac{2t}{\sqrt{n}}\Phi,\quad \mu_{X,t}=\mu_X+\frac{t}{\sqrt{n}}\xi_X,\quad \mu_{Y,t}=\mu_Y+\frac{t}{\sqrt{n}}\xi_Y,$$
$$\hat{\Sigma}=\frac{1}{2}\left(\frac{1}{n}\sum_{i=1}^n(X_i-\bar{X})(X_i-\bar{X})^T+\frac{1}{n}\sum_{i=1}^n(Y_i-\bar{Y})(Y_i-\bar{Y})^T\right),$$
\begin{equation}
V^2=4\left\|\Sigma^{*1/2}\Phi\Sigma^{*1/2}\right\|_F^2+\xi_X^T\Omega^*\xi_X+\xi_Y^T\Omega^*\xi_Y. \label{eq:LDAvariance}
\end{equation}
Assume $A_n$ is a set satisfying
$$A_n\subset\left\{\sqrt{n}\Big(||\mu_X-\mu_X^*||^2+||\mu_Y-\mu_Y^*||^2+||\Sigma-\Sigma^*||^2\Big)\right.\quad\quad\quad\quad$$
$$\left.\vee\sqrt{p}\Big(||\mu_X-\mu_X^*||+||\mu_Y-\mu_Y^*||+||\Sigma-\Sigma^*||\Big)\leq \delta_n\right\},\quad \text{for some }\delta_n=o(1),$$
The main result for LDA is the following theorem.
\begin{thm} \label{thm:LDA}
Assume that $||\Sigma^*||\vee||\Omega^*||= O(1)$, $p^2/n= o(1)$, and $V^{-1}= O(1)$. If for a given prior $\Pi$, the following two conditions are satisfied:
\begin{enumerate}
\item $\Pi(A_n|X^n,Y^n)=1-o_P(1)$,
\item For any fixed $t\in\mathbb{R}$, $\frac{\int_{A_n}\exp\Big(l_n(\mu_{X,t},\mu_{Y,t},\Omega_t)\Big)d\Pi(\mu_X,\mu_Y,\Omega)}{\int_{A_n}\exp\Big(l_n(\mu_X,\mu_Y,\Omega)\Big)d\Pi(\mu_X,\mu_Y,\Omega)}=1+o_P(1)$,
\end{enumerate}
then
$$\sup_{t\in\mathbb{R}}\left|\Pi\Big(\sqrt{n}V^{-1}\big(\Delta(\mu_X,\mu_Y,\Omega)-\hat{\Delta}\big)\leq t |X^n,Y^n\Big)-\mathbb{P}\Big(Z\leq t\Big)\right|=o_P(1),$$
where $Z\sim N(0,1)$ and the centering is $\hat{\Delta}=\Delta(\bar{X},\bar{Y},\hat{\Sigma}^{-1})$.
\end{thm}

A curious condition in the above theorem is $V^{-1}=O(1)$. The following proposition shows it is implied by the separation of the two classes.

\begin{proposition} \label{prop:V}
Under the setting of Theorem \ref{thm:LDA}, if $||\mu_X^*-\mu_Y^*||\geq c$ for some constant $c>0$, then we have $V^{-1}= O(1)$.
\end{proposition}
\begin{proof}
By the definition of $V^2$, we have
\begin{eqnarray*}
V^2 &\geq& \xi_X^T\Omega^*\xi_X + \xi_Y^T\Omega^*\xi_Y \geq C\Big(||\xi_X||^2+||\xi_Y||^2\Big) \\
&=& 4C\Big(||z-\mu_X^*||^2+||z-\mu_Y^*||^2\Big) \\
&\geq& 2C||\mu_X^*-\mu_Y^*||^2,
\end{eqnarray*}
which is greater than a constant under the separation assumption.
\end{proof}

Now we give examples of priors for LDA. Let us use independent priors. That is
$$\Omega\sim \Pi_{\Omega},\quad \mu_X\sim\Pi_X,\quad \mu_Y\sim\Pi_Y,$$
independently. The prior for the whole parameter $(\Omega,\mu_X,\mu_Y)$ is a product measure defined as
$$\Pi=\Pi_{\Omega}\times\Pi_X\times\Pi_Y.$$
Let $\Pi_{\Omega}$ be the Gaussian prior defined in (\ref{eq:gaussiandensity}). Let both $\Pi_X$ and $\Pi_Y$ be $N(0,I_{p\times p})$.

\begin{thm} \label{thm:LDAgauss}
Assume $||\Sigma^*||\vee||\Omega^*||\leq\Lambda=O(1)$, $V^{-1}= O(1)$, and $p^2=o\Big(\frac{\sqrt{n}}{\log n}\Big)$. The prior defined above satisfies the two conditions in Theorem \ref{thm:LDA} for some appropriate $A_n$. Thus, the Bernstein-von Mises result holds.
\end{thm}

\subsection{Quadratic Discriminant Analysis}

For the general case that $\Omega_X^*=\Omega_Y^*$ may not be true, the posterior distribution for QDA is defined as
$$\Pi(B|X^n)=\frac{\int_B\exp\Big(l_n(\mu_X,\mu_Y,\Omega_X,\Omega_Y)\Big)d\Pi(\mu_X,\mu_Y,\Omega_X,\Omega_Y)}{\int\exp\Big(l_n(\mu_X,\mu_Y,\Omega_X,\Omega_Y)\Big)d\Pi(\mu_X,\mu_Y,\Omega_X,\Omega_Y)},$$
where $l_n(\mu_X,\mu_X,\Omega_X,\Omega_Y)$ has decomposition
$$l_n(\mu_X,\mu_X,\Omega_X,\Omega_Y)=l_X(\mu_X,\Omega_X)+l_Y(\mu_Y,\Omega_Y).$$
We define the following quantities,
$$\Phi_X=-\Omega_X^*\Big(\Sigma_X^*-(z-\mu_X^*)(z-\mu_X^*)^T\Big)\Omega_X^*,$$
$$\Phi_Y=\Omega_Y^*\Big(\Sigma_Y^*-(z-\mu_Y^*)(z-\mu_Y^*)^T\Big)\Omega_Y^*,$$
$$\xi_X=2(z-\mu_X^*),\quad \xi_Y=2(\mu_Y^*-z),$$
$$\Omega_{X,t}=\Omega_X+\frac{2t}{\sqrt{n}}\Phi_X,\quad\Omega_{Y,t}=\Omega_Y+\frac{2t}{\sqrt{n}}\Phi_Y,\quad \mu_{X,t}=\mu_X+\frac{t}{\sqrt{n}}\xi_X,\quad \mu_{Y,t}=\mu_Y+\frac{t}{\sqrt{n}}\xi_Y,$$
$$\hat{\Sigma}_X=\frac{1}{n}\sum_{i=1}^n(X_i-\bar{X})(X_i-\bar{X})^T,\quad \hat{\Sigma}_Y=\frac{1}{n}\sum_{i=1}^n(Y_i-\bar{Y})(Y_i-\bar{Y})^T,$$
\begin{equation}
V^2=2\left\|\Sigma_X^{*1/2}\Phi_X\Sigma_X^{*1/2}\right\|_F^2+2\left\|\Sigma_Y^{*1/2}\Phi_Y\Sigma_Y^{*1/2}\right\|_F^2+\xi_X^T\Omega^*\xi_X+\xi_Y^T\Omega^*\xi_Y. \label{eq:QDAvariance}
\end{equation}
Assume $A_n$ is a set satisfying
\begin{eqnarray*}
A_n &\subset& \left\{\sqrt{n}\Big(||\mu_X-\mu_X^*||^2+||\mu_Y-\mu_Y^*||^2+||\Sigma-\Sigma^*||^2\Big)\right.\quad\quad\quad\quad \\
&& \left.\vee\sqrt{p}\Big(||\mu_X-\mu_X^*||+||\mu_Y-\mu_Y^*||+||\Sigma-\Sigma^*||\Big)\right.\\
&& \left.\vee\sqrt{n/p}\Big(||\Sigma_X-\Sigma_X^*||_F^2+||\Sigma_Y-\Sigma_Y^*||_F^2\Big)\right.\\
&& \left.\vee \sqrt{p}\Big(||\Sigma_X-\Sigma_X^*||_F+||\Sigma_Y-\Sigma_Y||_F\Big)\leq \delta_n\right\},
\end{eqnarray*}
with some $\delta_n=o(1)$.
The main result for QDA is the following theorem.
\begin{thm} \label{thm:QDA}
Assume $||\Sigma^*||\vee||\Omega^*|= O(1)$, $V^{-1}= O(1)$, and $p^3/n=o(1)$. If for a given prior $\Pi$, the following two conditions are satisfied:
\begin{enumerate}
\item $\Pi(A_n|X^n,Y^n)=1-o_P(1)$,
\item For any fixed $t\in\mathbb{R}$, $\frac{\int_{A_n}\exp\Big(l_n(\mu_{X,t},\mu_{Y,t},\Omega_{X,t},\Omega_{Y,t})\Big)d\Pi(\mu_X,\mu_Y,\Omega_X,\Omega_Y)}{\int_{A_n}\exp\Big(l_n(\mu_X,\mu_Y,\Omega_X,\Omega_Y)\Big)d\Pi(\mu_X,\mu_Y,\Omega_X,\Omega_Y)}=1+o_P(1)$,
\end{enumerate}
then
$$\sup_{t\in\mathbb{R}}\left|\Pi\Big(\sqrt{n}V^{-1}\big(\Delta(\mu_X,\mu_Y,\Omega_X,\Omega_Y)-\hat{\Delta}\big)\leq t |X^n,Y^n\Big)-\mathbb{P}\Big(Z\leq t\Big)\right|=o_P(1),$$
where $Z\sim N(0,1)$ and the centering is $\hat{\Delta}=\Delta(\bar{X},\bar{Y},\hat{\Sigma}_X^{-1},\hat{\Sigma}_Y^{-1})$.
\end{thm}

\begin{remark}
With the new definition of $V$ in QDA, the assumption $V^{-1}= O(1)$ is also implied by the separation condition $||\mu_X-\mu_Y||>c$ by applying the same argument in Proposition \ref{prop:V}.
\end{remark}

\begin{remark}
For independent prior in the sense that
$$d\Pi(\mu_X,\mu_Y,\Omega_X,\Omega_Y)=d\Pi(\mu_X,\Omega_X)\times d\Pi(\mu_Y,\Omega_Y),$$
the posterior is also independent because of the decomposition of the likelihood. In this case, we have
$$\Pi(A_n|X^n,Y^n)=\Pi_X(A_{X,n}|X^n)\times \Pi_Y(A_{Y,n}|Y^n),$$
with $A_{X,n}$ and $A_{Y,n}$ being versions of $A_n$ involving only $(\mu_X,\Omega_X)$ and $(\mu_Y,\Omega_Y)$. In the same way, we also have
\begin{eqnarray*}
&& \frac{\int_{A_n}\exp\Big(l_n(\mu_{X,t},\mu_{Y,t},\Omega_{X,t},\Omega_{Y,t})\Big)d\Pi(\mu_X,\mu_Y,\Omega_X,\Omega_Y)}{\int_{A_n}\exp\Big(l_n(\mu_X,\mu_Y,\Omega_X,\Omega_Y)\Big)d\Pi(\mu_X,\mu_Y,\Omega_X,\Omega_Y)} \\
&=& \frac{\int_{A_{X,n}}\exp\Big(l_n(\Omega_{X,t},\mu_{X,t})\Big)d\Pi_X(\mu_X,\Omega_X)}{\int_{A_{X,n}}\exp\Big(l_n(\Omega_X,\mu_X)\Big)d\Pi_X(\mu_X,\Omega_X)} \times \frac{\int_{A_{Y,n}}\exp\Big(l_n(\Omega_{Y,t},\mu_{Y,t})\Big)d\Pi_Y(\mu_Y,\Omega_Y)}{\int_{A_{Y,n}}\exp\Big(l_n(\Omega_Y,\mu_Y)\Big)d\Pi_Y(\mu_Y,\Omega_Y)}.
\end{eqnarray*}
Hence, for the two conditions in Theorem \ref{thm:QDA}, it is sufficient to check
\begin{enumerate}
\item $\Pi(A_{X,n}|X^n)=1-o_P(1)$,
\item For any fixed $t\in\mathbb{R}$, $\frac{\int_{A_{X,n}}\exp\Big(l_n(\Omega_{X,t},\mu_{X,t})\Big)d\Pi_X(\mu_X,\Omega_X)}{\int_{A_{X,n}}\exp\Big(l_n(\Omega_X,\mu_X)\Big)d\Pi_X(\mu_X,\Omega_X)}=1+o_P(1)$,
\end{enumerate}
and the corresponding conditions for $Y$, when the prior has an independent structure.
\end{remark}

The example of prior we specify for QDA is similar to the one for LDA. Let us use independent priors. That is
$$\Omega_X\sim \Pi_{\Omega_X},\quad\Omega_Y\sim\Pi_{\Omega_Y},\quad \mu_X\sim\Pi_X,\quad \mu_Y\sim\Pi_Y,$$
independently. The prior for the whole parameter $(\Omega_X,\Omega_Y,\mu_X,\mu_Y)$ is a product measure defined as
$$\Pi=\Pi_{\Omega_X}\times\Pi_{\Omega_Y}\times\Pi_X\times\Pi_Y.$$
Let $\Pi_{\Omega_X}$ and $\Pi_{\Omega_Y}$ be the Gaussian prior defined in Section \ref{sec:gaussprior}. Let both $\Pi_X$ and $\Pi_Y$ be $N(0,I_{p\times p})$.

\begin{thm} \label{thm:QDAgauss}
Assume $||\Sigma_X^*||\vee||\Omega_X^*||\vee||\Sigma^*_Y||\vee||\Omega^*_Y||\leq\Lambda=O(1)$, $V^{-1}= O(1)$ and $p^2=o\Big(\frac{\sqrt{n}}{\log n}\Big)$. The prior defined above satisfies the two conditions in Theorem \ref{thm:LDA} for some appropriate $A_n$. Thus, the Bernstein-von Mises result holds.
\end{thm}

\section{Discussion} \label{sec:disc}

\subsection{Comparison: Asymptotic Normality of $\phi(\hat{\Sigma})$ and $\psi(\hat{\Sigma}^{-1})$}

In this section, we present the classical  results for asymptotic normality of the estimators $\phi(\hat{\Sigma})$ and $\psi(\hat{\Sigma}^{-1})$. Note that in many cases, they coincide with  MLE. The purpose is to compare them with the BvM results obtained in this paper. We first review and define some notation. Remember $\hat{\sigma}_{ij}$ is the $(i,j)$-th element of $\hat{\Sigma}$ and $\hat{\omega}_{ij}$ is the $(i,j)$-th element of $\hat{\Sigma}^{-1}$. We let $\Delta_L$ and $\Delta_Q$ be the LDA and QDA functionals respectively. The corresponding asymptotic variances are denoted by $V_L^2$ and $V_Q^2$, defined in (\ref{eq:LDAvariance}) and (\ref{eq:QDAvariance}) respectively. As $p,n\rightarrow\infty$ jointly, the asymptotic normality of $\phi(\hat{\Sigma})$ or $\psi(\hat{\Sigma}^{-1})$ holds under different asymptotic regimes for different functionals. For comparison, we assume that $V_L$, $V_Q$ and the eigengap $\delta$ are at constant levels.

\begin{thm}
Let $p,n\rightarrow\infty$ jointly, then for any asymptotic regime of $(p,n)$,
$$\frac{\sqrt{n}(\hat{\sigma}_{ij}-\sigma_{ij}^*)}{\sqrt{\sigma_{ij}^{*2}+\sigma_{ii}^*\sigma_{jj}^*}}\leadsto N(0,1),$$
$$\frac{\sqrt{n}(v^T\hat{\Sigma}v-v^T\Sigma^*v)}{\sqrt{2}|v^T\Sigma^*v|}\leadsto N(0,1).$$
Assume $p^2/n=o(1)$, we have
\begin{equation}
\frac{\sqrt{n}(\hat{\omega}_{ij}-\omega_{ij}^*)}{\sqrt{\omega_{ij}^{*2}+\omega_{ii}^*\omega_{jj}^*}}\leadsto N(0,1), \label{eq:ANomega}
\end{equation}
\begin{equation}
\frac{\sqrt{n}(v^T\hat{\Sigma}^{-1}v-v^T\Omega^*v)}{\sqrt{2}|v^T\Omega^*v|}\leadsto N(0,1), \label{eq:ANquadomega}
\end{equation}
\begin{equation}
\frac{\sqrt{n}\big(\lambda_m(\hat{\Sigma})-\lambda_m(\Sigma^*)\big)}{\sqrt{2}\lambda_m(\Sigma^*)}\leadsto N(0,1), \label{eq:ANeigensigma}
\end{equation}
\begin{equation}
\frac{\sqrt{n}\big(\lambda_m(\hat{\Sigma}^{-1})-\lambda_m(\Omega^*)\big)}{\sqrt{2}\lambda_m(\Omega^*)}\leadsto N(0,1), \label{eq:ANeigenomega}
\end{equation}
\begin{equation}
\sqrt{n}V_L^{-1}\Big(\Delta_L(\bar{X},\bar{Y},\hat{\Sigma}^{-1})-\Delta_L(\mu_X^*,\mu_Y^*,\Omega^*)\Big)\leadsto N(0,1). \label{eq:ANLDA}
\end{equation}
Assume $p^3/n=o(1)$, we have
\begin{equation}
\sqrt{\frac{n}{2p}}\Big(\log\det(\hat{\Sigma})-\log\det(\Sigma^*)\Big)\leadsto N(0,1), \label{eq:ANlogdet}
\end{equation}
\begin{equation}
\sqrt{n}V_Q^{-1}\Big(\Delta_Q(\bar{X},\bar{Y},\hat{\Sigma}_X^{-1},\hat{\Sigma}_Y^{-1})-\Delta_Q(\mu_X^*,\mu_Y^*,\Omega_X^*,\Omega_Y^*)\Big)\leadsto N(0,1). \label{eq:ANQDA}
\end{equation}
\end{thm}

Since the above results are more or less scattered in the literature, we do not present their proofs in this paper. Readers who are interested can derive these results using delta method.

We remark that the condition $p^2/n=o(1)$ is sharp for (\ref{eq:ANomega})-(\ref{eq:ANLDA}). For (\ref{eq:ANomega}) and (\ref{eq:ANquadomega}), a common example is $\omega_{11}=e_1^T\Omega e_1$, where $e_1^T=(1,0,...,0)$. By distributional facts of inverse Wishart, $\sqrt{n}(\hat{\omega}_{11}-\omega_{11}^*)$ is not asymptotically normal if $p^2/n=o(1)$ does not hold. Since the functional $\Delta_L$ is harder than $v^T\Omega v$ (the latter is a special case of the former if $\mu_X^*$ and $\mu_Y^*$ are known), $p^2/n=o(1)$ is also sharp for (\ref{eq:ANLDA}). For (\ref{eq:ANeigensigma}) and (\ref{eq:ANeigenomega}), we have the following proposition to show that $p^2/n=o(1)$ is necessary.

\begin{proposition} \label{prop:eigencounter}
Consider a diagonal $\Sigma^*$. Let the eigengap $\sigma_{11}^*-\sigma_{22}^*$ be at constant level when $p,n\rightarrow\infty$ jointly. Assume $||\Sigma^*||\vee||\Omega^*||= O(1)$, $n^{1/2}= o(p)$ and $p= o(n^{2/3})$. Then $\lambda_1(\hat{\Sigma})$ is not $\sqrt{n}$-consistent. As a consequence, $\lambda_p(\hat{\Sigma}^{-1})=\lambda_1^{-1}(\hat{\Sigma})$ is not $\sqrt{n}$-consistent.
\end{proposition}

The condition $p^3/n=o(1)$ is sharp for (\ref{eq:ANlogdet}) and (\ref{eq:ANQDA}). If $p^3/n=o(1)$ does not hold, a bias correction is necessary for (\ref{eq:ANlogdet}) to hold (see \cite{cai13}). That the condition $p^3/n=o(1)$ is  necessary for (\ref{eq:ANQDA}) is because the functional $\Delta_Q$ contains the part $\log\det(\Sigma)$.

In the next section, we are going to discuss the asymptotic regime of $(p,n)$ for BvM and compare them with the frequentist results listed in this section.

\subsection{The Asymptotic Regime of $(p,n)$} \label{sec:pn}

For all the BvM results we obtain in this paper, they assume different asymptotic regime of the sample size $n$ and the dimension $p$. Ignoring the $\log n$ factor and assume constant eigengap $\delta$ and asymptotic variances for LDA and QDA, the asymptotic regime for $(p,n)$ is summarized in the following table.
\begin{center}
    \begin{tabular}{| l | l | l | l |}
    \hline
    functional & $\phi(\hat{\Sigma})$ or $\psi(\hat{\Sigma}^{-1})$ & conjugate & non-conjugate \\ \hline
    $\sigma_{ij}$ & *** & $p\ll n$ & $p^2\ll n$  \\ \hline
    $\omega_{ij}$ & $p^2\ll n$ & $p^2\ll n$ & $p^3\ll n$  \\ \hline
    $v^T\Sigma v$ & *** & $p\ll n$  & $p^2\ll n$ \\ \hline
    $v^T\Omega v$ & $p^2\ll n$ & $p^2\ll n$ & $p^3\ll n$ \\ \hline
    $\log\det(\Sigma)$ & $p^3\ll n$ & $p^3\ll n$ & $p^3\ll n$ \\ \hline
    $\lambda_m(\Sigma)$ & $p^2\ll n$ & $p^2\ll n$ & $p^4\ll n$ \\ \hline
    $\lambda_m(\Omega)$ & $p^2\ll n$ & $p^2\ll n$ & $p^4\ll n$ \\ \hline
    LDA & $p^2\ll n$ & $p^2\ll$ n & $p^4\ll n$ \\ \hline
    QDA & $p^3\ll n$ & $p^3\ll$ n & $p^4\ll n$ \\ \hline
    \end{tabular}
\end{center}

The table has three columns for the asymptotic normality of $\phi(\hat{\Sigma})$ and $\psi(\hat{\Sigma}^{-1})$ and for BvM with conjugate and non-conjugate priors respectively. The purpose is to compare our BvM result with the classical frequentist asymptotic normality. The priors are the Wishart prior and Gaussian prior we consider in this paper. For discriminant analysis, we did not consider conjugate prior because of limit of space. The conjugate prior in the LDA and QDA settings is the normal-Wishart prior. Its posterior distribution can be decomposed as a marginal Wishart times a conditional normal. The analysis of the BvM result for this case is direct, and we claim the asymptotic regimes for LDA and QDA are $p^2\ll n$ and $p^3\ll n$ respectively without giving a formal proof.

Comparing the first and the second columns, the condition for $p$ and $n$ we need for the BvM results with conjugate prior matches the conditions for the frequentist results. The two exceptions are $\sigma_{ij}$ and $v^T\Sigma v$, where for the frequentist asymptotic normality to hold, there is no assumption on $p,n$. Our technique of proof requires $p\ll n$. This is because our theory requires a set $A_n\subset \{||\Sigma-\Sigma^*||\leq\delta_n\}$ for some $\delta_n=o(1)$ to satisfy $\Pi(A_n|X^n)=1-o_P(1)$. The best rate of convergence for $||\Sigma-\Sigma^*||$ is $\sqrt{p/n}$, which leads to $p\ll n$. Such assumption may be weaken if a different theory than ours can be developed (or through direct calculation by taking advantage of the conjugacy).

The comparison of the second and the third columns suggests that using of non-conjugate prior requires stronger assumptions. We believe these stronger assumptions can all be weakened. The current stronger assumptions on $p$ and $n$ are caused  the technique we use in this paper to prove posterior contraction, which is Condition 1 in  Theorem \ref{thm:main1} and Theorem \ref{thm:main2}.
The current way of proving posterior contraction in nonparametric Bayes theory only allows loss functions which are at the same order of the Kullback-Leibler divergence. In the covariance matrix estimation setting, we can only deal with Frobenius loss. We choose
$$A_n=\left\{||\Sigma-\Sigma^*||_F^2\leq M\frac{p^2\log n}{n}\right\}.$$
For functionals of covariance such as $\sigma_{ij}$ and $v^T\Sigma v$, we need $A_n\subset\{||\Sigma-\Sigma^*||\leq\delta_n\}$ for some $\delta_n$. We have to bound $||\Sigma-\Sigma^*||$ as
$$||\Sigma-\Sigma^*||\leq ||\Sigma-\Sigma^*||_F\leq \sqrt{M\frac{p^2\log n}{n}},$$
and require $\sqrt{M\frac{p^2\log n}{n}}\leq\delta_n=o(1)$. This leads to $p^2\ll n$. For functionals of precision matrix, we need $A_n\subset\{\sqrt{p}||\Sigma-\Sigma^*||\leq\delta_n\}$. Again, we have bound
$$\sqrt{p}||\Sigma-\Sigma^*||\leq \sqrt{p}||\Sigma-\Sigma^*||_F\leq \sqrt{M\frac{p^3\log n}{n}},$$
and require $\sqrt{M\frac{p^3\log n}{n}}=o(1)$. This leads to $p^3\ll n$. It would be great if we can prove a posterior contraction on $\{||\Sigma-\Sigma^*||\leq M\sqrt{p/n}\}$ directly without referring to the Frobenius loss. However, under the current technique of Bayes nonparemtrics \cite{ghosal00}, this is impossible. See a lower bound argument in \cite{hoffmann13}.

\subsection{Sharpness of The Condition $rp^2/n=o(1)$ in Theorem \ref{thm:main2}} \label{sec:sharpthm2}

It is curious whether the condition $rp^2/n=o(1)$ is sharp in Theorem \ref{thm:main2}. Let us consider the funcitonal $\psi(\Omega)=\log\det(\Omega)$. In this case, the corresponding matrix $\Psi$ in the linear expansion of $\psi(\Omega)$ is $\Psi=\Sigma^*$ and $r=\text{rank}(\Sigma^*)=p$. Then, the condition $rp^2=o(1)$ becomes $p^3/n=o(1)$.
Since $\log\det(\Omega)=-\log\det(\Sigma)$ and $p^3/n=o(1)$ is sharp for  BvM to hold for $\log\det(\Sigma)$, it is also sharp for $\log\det(\Omega)$.

\subsection{Covariance Priors}

The general framework in Section \ref{sec:framework} only considers prior defined on precision matrix $\Omega$. However, sometimes it is more natural to use prior defined on covariance matrix $\Sigma$, for example, Gaussian prior on $\Sigma$. Then, the first conditions in Theorem \ref{thm:main1} and Theorem \ref{thm:main2} are hard to check. We propose a slight variation of this condition, so that our theory can also be user-friendly for covariance priors.

We first consider approximate linear functionals of $\Sigma$ satisfying (\ref{eq:approxlincov}). Then, the first condition of Theorem \ref{thm:main1} can be replaced by
$$
\frac{\int_{A_n}\exp\Big(l_n(\Sigma_t^{-1})\Big)d\Pi(\Sigma)}{\int_{A_n}\exp\Big(l_n(\Sigma^{-1})\Big)d\Pi(\Sigma)}=1+o_P(1),\quad\text{for each fixed } t\in\mathbb{R}, $$
where $\Sigma_t=\Sigma-\frac{2t}{\sqrt{n}||\Sigma^{*1/2}\Phi\Sigma^{*1/2}||_F}\Sigma^*\Phi\Sigma^*$. Then we consider approximate linear functionals of $\Omega$ satisfying (\ref{eq:approxlinprecision}). The first condition of Theorem \ref{thm:main2} can be replaced by
$$\frac{\int_{A_n}\exp\Big(l_n(\Sigma_t^{-1})\Big)d\Pi(\Sigma)}{\int_{A_n}\exp\Big(l_n(\Sigma^{-1})\Big)d\Pi(\Sigma)}=1+o_P(1),\quad\text{for each fixed } t\in\mathbb{R},$$
where $\Sigma_t=\Sigma+\frac{2t}{\sqrt{n}||\Omega^*{1/2}\Psi\Omega^{*1/2}||_F}\Psi$.

With the new conditions, it is direct to check them for covariance priors by change of variable, as is done in the proof of Lemma \ref{lem:wishart} and Lemma \ref{lem:gaussprior}. In particular, for the Gaussian prior on covariance matrix, we claim the conclusion of Lemma \ref{lem:gaussprior} holds. We avoid expanding the technical details for the covariance priors in this paper due to the limit of space.

\subsection{Relation to Matrix Estimation under Non-Frobenius Loss}

As we have mentioned in the end of Section \ref{sec:pn}, the current Bayes nonparametric technique for proving posterior contraction rate only covers losses which are at the same order of Kullback-Leiber divergence. It cannot handle other non-intrinsic loss \citep{hoffmann13}. In the Bayes matrix estimation setting, whether we can show the following conclusion
\begin{equation}
\Pi\Big(||\Sigma-\Sigma^*||\leq M\sqrt{\frac{p}{n}}\Big|X^n\Big)=1-o_P(1), \label{eq:futurework}
\end{equation}
for a general non-conjugate prior still remains open. This explains why there is so little literature in this field compared to the growing research using frequentist methods. See, for example, \cite{cai10} and \cite{cai12}.

However, we observe that for the spectral norm loss,
$$||\Sigma-\Sigma^*||\leq 2\sup_{v\in\mathcal{N}}|v^T(\Sigma-\Sigma^*)v|,$$
where $\mathcal{N}$ is a subset of $S^{p-1}$ with cardinality bound $\log|\mathcal{N}|\leq cp$ for some $c>0$. The BvM result we establish for the functional $v^T\Sigma v$ indicates that for each $v$, the posterior distribution of $|v^T(\Sigma-\Sigma^*)v|$ is at the order of $n^{-1/2}$. Therefore, heuristically, $2\sup_{v\in\mathcal{N}}|v^T(\Sigma-\Sigma^*)v|$ should be at the order of $\frac{\sqrt{\log|\mathcal{N}|}}{\sqrt{n}}$, which is $\sqrt{p/n}$. We will use this intuition as a key idea in our future research project on the topic of Bayes matrix estimation.

Once (\ref{eq:futurework}) is established for a non-conjugate prior (e.g. Gaussian prior in this paper), then we may use (\ref{eq:futurework}) to weaken the conditions in the third column of the table in Section \ref{sec:pn}. In fact, most entries of that column can be weakened to match the conditions in the second column for a conjugate prior. As argued in Section \ref{sec:pn}, (\ref{eq:futurework}) directly implies the concentration $\Pi(A_n|X^n)=1-o_P(1)$, which is Condition 1 in both Theorem \ref{thm:main1} and Theorem \ref{thm:main2}.

\section{Proofs} \label{sec:proof}

\subsection{Proof of Theorem \ref{thm:main1} \& Theorem \ref{thm:main2}}

Before stating the proofs, we first display some lemmas.
The following lemma is Lemma 2 in \cite{castillo13b}. It allows us to prove BvM results through convergence of moment generating functions.

\begin{lemma} \label{lem:laplace}
Consider the random probability measure $P_n$ and a fixed probability measure $P$. Suppose for any real $t$, the Laplace transformation $\int e^{tx}dP(x)$ is finite, and $\int e^{tx}dP_n(x)\rightarrow \int e^{tx}dP(x)$ in probability. Then, it holds that
$$\sup_{t\in\mathbb{R}}\left|P_n\Big((-\infty,t]\Big)-P\Big((-\infty,t]\Big)\right|=o_P(1).$$
\end{lemma}

The next lemma is an expansion of the Gaussian likelihood.
\begin{lemma} \label{lem:likexpansion}
Assume $||\Sigma^*||\vee||\Omega^*||= O(1)$.
For any symmetric matrix $\Phi$  and the perturbed precision matrix
$$\Omega_t=\Omega+\frac{\sqrt{2}t}{\sqrt{n}\left\|\Sigma^{*1/2}\Phi\Sigma^{*1/2}\right\|_F}\Phi,$$
the following equation holds for all $\Omega\in A_n$ with $A_n$ satisfying (\ref{eq:covsubset}) or (\ref{eq:precisionsubset}).
\begin{equation}
l_n(\Omega_t)-l_n(\Omega)=\frac{t\sqrt{n}}{\sqrt{2}\left\|\Sigma^{*1/2}\Phi\Sigma^{*1/2}\right\|_F}\text{tr}\Big((\Sigma-\hat{\Sigma})\Phi\Big)-\frac{1}{2}t^2\frac{\left\|\Sigma^{1/2}\Phi\Sigma^{1/2}\right\|_F^2}{\left\|\Sigma^{*1/2}\Phi\Sigma^{*1/2}\right\|_F^2}-\frac{n}{2}\sum_{j=1}^p\frac{(h_j-s)^2}{(1-s)^3}ds, \label{eq:likexpansion}
\end{equation}
where $\{h_j\}_{j=1}^p$ are eigenvalues of $\Sigma^{1/2}(\Omega-\Omega_t)\Sigma^{1/2}$.
\end{lemma}

The following lemma is Proposition D.1 in the supplementary material of \cite{ma13}, which is rooted in \cite{davidson01}.
\begin{lemma} \label{lem:wishartconcen}
Let $Y_l\sim N(0,I_{p\times p})$. Then, for any $t>0$,
$$P_{I}^n\Bigg(\left\|\frac{1}{n}\sum_{l=1}^nY_lY_l^T-I\right\|\leq 2\Big(\sqrt{\frac{p}{n}}+t\Big)+\Big(\sqrt{\frac{p}{n}}+t\Big)^2\Bigg)\geq 1-2e^{-nt^2/2}.$$
\end{lemma}

\begin{proof}[Proof of Theorem \ref{thm:main1}]
We are going to use Lemma \ref{lem:laplace} and establish the convergence of moment generating function.
We claim that
\begin{equation}
l_n(\Omega_t)-l_n(\Omega)=\frac{t\sqrt{n}}{\sqrt{2}\left\|\Sigma^{*1/2}\Phi\Sigma^{*1/2}\right\|_F}\Big(\phi(\Sigma)-\phi(\hat{\Sigma})\Big)-\frac{1}{2}t^2+o_P(1), \label{eq:covexpansion1}
\end{equation}
uniformly over $A_n$. The derivation of (\ref{eq:covexpansion1}) will be given at the end of the proof.  Define the posterior distribution conditioning on $A_n$ by
$$\Pi^{A_n}(B|X^n)=\frac{\Pi(A_n\cap B|X^n)}{\Pi(A_n|X^n)},\quad\text{for any }B.$$
It is easy to see
\begin{equation}
\sup_B\left|\Pi^{A_n}(B|X^n)-\Pi(B|X^n)\right|=o_P(1),\label{eq:condition}
\end{equation}
by the first condition of Theorem \ref{thm:main1}. Now we calculation the moment generating function of $\frac{\sqrt{n}\Big(\phi(\Sigma)-\phi(\hat{\Sigma})\Big)}{\sqrt{2}\left\|\Sigma^{*1/2}\Phi\Sigma^{*1/2}\right\|_F}$ under the distribution $\Pi^{A_n}(\cdot|X^n)$, which is
\begin{eqnarray*}
&& \int \exp\Bigg(\frac{t\sqrt{n}\Big(\phi(\Sigma)-\phi(\hat{\Sigma})\Big)}{\sqrt{2}\left\|\Sigma^{*1/2}\Phi\Sigma^{*1/2}\right\|_F}\Bigg)d\Pi^{A_n}(\Omega|X^n) \\
&=& \frac{\int_{A_n}\exp\Big(\frac{t\sqrt{n}\Big(\phi(\Sigma)-\phi(\hat{\Sigma})\Big)}{\sqrt{2}\left\|\Sigma^{*1/2}\Phi\Sigma^{*1/2}\right\|_F}+l_n(\Omega)\Big)d\Pi(\Omega)}{\int_{A_n}\exp\Big(l_n(\Omega)\Big)d\Pi(\Omega)} \\
&=& \Big(1+o_P(1)\Big)\exp\big(t^2/2\big)\frac{\int_{A_n}\exp\Big(l_n(\Omega_t)\Big)d\Pi(\Omega)}{\int_{A_n}\exp\Big(l_n(\Omega)\Big)d\Pi(\Omega)} \\
&=& \Big(1+o_P(1)\Big)\exp\big(t^2/2\big),
\end{eqnarray*}
where the second equality is because of (\ref{eq:covexpansion1}) and the last inequality is because of the second condition of Theorem \ref{thm:main1}. We have shown that the moment generating function of $\frac{\sqrt{n}\Big(\phi(\Sigma)-\phi(\hat{\Sigma})\Big)}{\sqrt{2}\left\|\Sigma^{*1/2}\Phi\Sigma^{*1/2}\right\|_F}$ under the distribution $\Pi^{A_n}(\cdot|X^n)$ converges to the moment generating function of $N\big(0,1\big)$ in probability. By Lemma \ref{lem:laplace} and (\ref{eq:condition}), we have established the desired result.

To finish the proof, let us derive (\ref{eq:covexpansion1}).
 Using the result of the likelihood expansion in Lemma \ref{lem:likexpansion}, we will first show
\begin{equation}
l_n(\Omega_t)-l_n(\Omega)=\frac{t\sqrt{n}}{\sqrt{2}\left\|\Sigma^{*1/2}\Phi\Sigma^{*1/2}\right\|_F}\text{tr}\Big((\Sigma-\hat{\Sigma})\Phi\Big)-\frac{t^2}{2}+o(1), \label{eq:covexpansion}
\end{equation}
where the $o(1)$ above is uniform on $A_n$. Compare (\ref{eq:covexpansion}) with (\ref{eq:likexpansion}) in Lemma \ref{lem:likexpansion}, it is sufficient to bound
$$R_1=\left|\frac{\left\|\Sigma^{1/2}\Phi\Sigma^{1/2}\right\|_F^2}{\left\|\Sigma^{*1/2}\Phi\Sigma^{*1/2}\right\|_F^2}-1\right|\quad\text{and}\quad R_2=\left|\frac{n}{2}\sum_{j=1}^p\frac{(h_j-s)^2}{(1-s)^3}ds\right|.$$
We use the following argument to bound $R_1$ on $A_n$.
\begin{eqnarray}
\nonumber && \left|||\Sigma^{1/2}\Phi\Sigma^{1/2}||_F^2-||\Sigma^{1/2*}\Phi\Sigma^{*1/2}||_F^2\right| \\
\nonumber &=& \left|\text{tr}(\Sigma\Phi\Sigma\Phi)-\text{tr}(\Sigma^*\Phi\Sigma^*\Phi)\right| \\
\nonumber &\leq& \left|\text{tr}\Big(\Sigma\Phi(\Sigma-\Sigma^*)\Phi\Big)\right| + \left|\text{tr}\Big((\Sigma-\Sigma^*)\Phi\Sigma^*\Phi\Big)\right| \\
\nonumber &=& \left|\text{tr}\Big(\Sigma^{1/2}\Phi\Sigma\Phi\Sigma^{1/2}(I-\Sigma^{-1/2}\Sigma^*\Sigma^{-1/2})\Big)\right|  \\
\nonumber && + \left|\text{tr}\Big(\Sigma^{*1/2}\Phi\Sigma^*\Phi\Sigma^{*1/2}(\Sigma^{*-1/2}\Sigma\Sigma^{*-1/2}-I)\Big)\right| \\
\label{eq:VNtrace} &\leq& \text{tr}\Big(\Sigma^{1/2}\Phi\Sigma\Phi\Sigma^{1/2}\Big)||I-\Sigma^{-1/2}\Sigma^*\Sigma^{-1/2}|| \\
\nonumber && + \text{tr}\Big(\Sigma^{*1/2}\Phi\Sigma^*\Phi\Sigma^{*1/2}\Big)||I-\Sigma^{*-1/2}\Sigma\Sigma^{*-1/2}|| \\
\nonumber &=&  ||\Sigma^{1/2}\Phi\Sigma^{1/2}||_F^2||I-\Sigma^{-1/2}\Sigma^*\Sigma^{-1/2}||+||\Sigma^{1/2*}\Phi\Sigma^{*1/2}||_F^2||I-\Sigma^{*-1/2}\Sigma\Sigma^{*-1/2}||  \\
\label{eq:byAn} &\leq&  o(1)||\Sigma^{1/2}\Phi\Sigma^{1/2}||_F^2+o(1)||\Sigma^{1/2*}\Phi\Sigma^{*1/2}||_F^2,
\end{eqnarray}
where the inequality (\ref{eq:VNtrace}) is by von Neumann's trace inequality and the inequality (\ref{eq:byAn}) is due to the fact that  $||\Sigma-\Sigma^*||= o(1)$ on $A_n$. Rearranging the above argument, we get $R_1=o(1)$ uniformly on $A_n$.
To bound $R_2$, we first use Weyl's theorem to get
$$\max_{1\leq j\leq p}|h_j|\leq \frac{\sqrt{2}t}{\sqrt{n}}\frac{\left\|\Sigma^{1/2}\Phi\Sigma^{1/2}\right\|}{\left\|\Sigma^{*1/2}\Phi\Sigma^{*1/2}\right\|_F}= O(n^{-1/2}),$$
on $A_n$.
Thus, on $A_n$, we have
\begin{eqnarray*}
R_2 &\leq& Cn\left|\sum_{j=1}^p\int_0^{h_j}(h_j-s)^2ds\right| \\
 &\leq& Cn\sum_{j=1}^p|h_j|^3 \\
 &\leq& Cn\max_{1\leq j\leq p}|h_j| \sum_{j=1}^p|h_j|^2 \\
 &\leq& Cn \times O(n^{-1/2}) \times O\Bigg(\frac{\left\|\Sigma^{1/2}\Phi\Sigma^{1/2}\right\|_F^2}{n\left\|\Sigma^{*1/2}\Phi\Sigma^{*1/2}\right\|_F^2}\Bigg) \\
 &=& O(n^{-1/2}).
\end{eqnarray*}
Hence, (\ref{eq:covexpansion}) is proved. Together with the approximate linearity condition (\ref{eq:approxlincov}) of the functional $\phi(\Sigma)$, (\ref{eq:covexpansion1}) is proved. Thus, the proof is complete.
\end{proof}

\begin{proof}[Proof of Theorem \ref{thm:main2}]
We follow the reasoning in the proof of Theorem \ref{thm:main1} and omit some similar steps. Define
$$\Phi=-\Omega^*\Psi\Omega^*.$$
It is easy to see that
$$\left\|\Omega^{*1/2}\Psi\Omega^{*1/2}\right\|_F=\left\|\Sigma^{*1/2}\Phi\Sigma^{*1/2}\right\|_F.$$
Then by Lemma \ref{lem:likexpansion} and the similar arguments in the proof of Theorem \ref{thm:main1}, we obtain
$$l_n(\Omega_t)-l_n(\Omega)=\frac{t\sqrt{n}\text{tr}\Big((\Sigma-\hat{\Sigma})\Phi\Big)}{\sqrt{2}\left\|\Omega^{*1/2}\Psi\Omega^{*1/2}\right\|_F}-\frac{1}{2}t^2+o(1),$$
uniformly on $A_n$, which is analogous to (\ref{eq:covexpansion}). 
We are going to approximate $\sqrt{n}\text{tr}\Big((\Sigma-\hat{\Sigma})\Phi\Big)$ by $\sqrt{n}\Big(\psi(\Omega)-\psi(\hat{\Sigma}^{-1})\Big)$ on $A_n$. Define $\hat{\Omega}=\hat{\Sigma}^{-1}$. The assumption $rp^2/n=o(1)$ implies that $p/n=o(1)$. Thus, $\hat{\Omega}$ is well defined. By Lemma \ref{lem:wishartconcen},
\begin{equation}
||\hat{\Sigma}-\Sigma^*||= O\Big(\sqrt{\frac{p}{n}}\Big),\quad ||\hat{\Omega}-\Omega^*||= O\Big(\sqrt{\frac{p}{n}}\Big). \label{eq:samplecovconverge}
\end{equation}
Using notation $V=2\left\|\Omega^{*1/2}\Psi\Omega^{*1/2}\right\|_F^2$, the approximation error on $A_n$ is
\begin{eqnarray*}
&& \sqrt{n}V^{-1/2}\left|\psi(\Omega)-\psi(\hat{\Omega})-\text{tr}\Big((\Sigma-\hat{\Sigma})\Phi\Big)\right| \\
&=& \sqrt{n}V^{-1/2}\left|\text{tr}\Big((\Omega-\hat{\Omega})\Psi\Big)+\text{tr}\Big((\Sigma-\hat{\Sigma})\Omega^*\Psi\Omega^*\Big)\right| \\
&=& \sqrt{n}V^{-1/2}\left|\text{tr}\Big((\Sigma-\hat{\Sigma})(\Omega^*\Psi\Omega^*-\Omega\Psi\hat{\Omega})\Big)\right| \\
&\leq& \sqrt{n}V^{-1/2}\left|\text{tr}\Big((\Sigma-\hat{\Sigma})\Omega^*\Psi(\Omega^*-\hat{\Omega})\Big)\right| + \sqrt{n}V^{-1/2}\left|\text{tr}\Big((\Sigma-\hat{\Sigma})(\Omega^*-\Omega)\Psi\hat{\Omega}\Big)\right|.
\end{eqnarray*}
Let the singular value decomposition of $\Psi$ be $\Psi=\sum_{l=1}^r d_l q_lq_l^T$. Then,
\begin{eqnarray*}
&& \left|\text{tr}\Big((\Sigma-\hat{\Sigma})\Omega^*\Psi(\Omega^*-\hat{\Omega})\Big)\right| \\
&\leq& \sum_{l=1}^r |d_l| \left|\text{tr}\Big((\Sigma-\hat{\Sigma})\Omega^*q_lq_l^T(\Omega^*-\hat{\Omega})\Big)\right| \\
&\leq&  \sum_{l=1}^r |d_l| ||(\Sigma-\hat{\Sigma})\Omega^*q_l|| ||q_l^T(\Omega^*-\hat{\Omega})\Big)|| \\
&\leq& O_P\Bigg(||\hat{\Sigma}-\Sigma||||\hat{\Sigma}-\Sigma^*||\sum_{l=1}^r|d_l|\Bigg).
\end{eqnarray*}
Similarly, 
\begin{eqnarray*}
&& \left|\text{tr}\Big((\Sigma-\hat{\Sigma})(\Omega^*-\Omega)\Psi\hat{\Omega}\Big)\right| \\
&\leq& O_P\Bigg(||\hat{\Sigma}-\Sigma||||\hat{\Sigma}-\Sigma^*||\sum_{l=1}^r|d_l|\Bigg).
\end{eqnarray*}
Since
$$V^{-1/2}\leq C||\Psi||_F^{-1}=\frac{C}{\sqrt{\sum_{l=1}^rd_l^2}},$$
we have
\begin{eqnarray*}
&& \sqrt{n}V^{-1/2}\left|\psi(\Omega)-\psi(\hat{\Omega})-\text{tr}\Big((\Sigma-\hat{\Sigma})\Phi\Big)\right| \\
&\leq& O_P\Big(\sqrt{nr}||\hat{\Sigma}-\Sigma||||\hat{\Sigma}-\Sigma^*||\Big) \\
&\leq& O_P\Big(\sqrt{nr}||\hat{\Sigma}-\Sigma^*||^2+\sqrt{nr}||\hat{\Sigma}-\Sigma^*||||\Sigma-\Sigma^*||\Big) \\
&\leq& O_P\Bigg(\sqrt{\frac{rp^2}{n}}+\sqrt{rp}||\Sigma-\Sigma^*||\Bigg) \\
&=& o_P(1)
\end{eqnarray*}
uniformly on $A_n$, where we have used (\ref{eq:samplecovconverge}) in the second last inequality above. Hence,
\begin{equation}
l_n(\Omega_t)-l_n(\Omega)=\frac{t\sqrt{n}\Big(\psi(\Omega)-\psi(\hat{\Sigma}^{-1})\Big)}{\sqrt{2}\left\|\Omega^{*1/2}\Psi\Omega^{*1/2}\right\|_F}-\frac{1}{2}t^2+o_P(1),
\end{equation}
uniformly on $A_n$. The remaining part of the proof are the same as the corresponding steps in the proof of Theorem \ref{thm:main1}. Thus, the proof is complete.
\end{proof}

\subsection{Proof of Lemma \ref{lem:wishart}} \label{sec:proofwishart}

\begin{proof}[Proof of Lemma \ref{lem:wishart}]
The proof has two parts. In the first part, we establish the first condition of the two theorems by proving a posterior contraction rate. In the second part, we establish the second condition of the two theorems by showing that a change of variable is negligible under Wishart density.

\noindent\textbf{Part I. }
The posterior distribution $\Omega|X^n$ is $\mathcal{W}_p\Big((n\hat{\Sigma}+I)^{-1},n+p+b-1\Big)$. Conditioning on $X^n$, let $Z_l|X^n\sim P_{(n\hat{\Sigma}+I)^{-1}}$ i.i.d. for each $l=1,2,...,n+p+b-1$. Then the posterior distribution of $\Omega$ is identical to the distribution of $\sum_{l=1}^{n+p+b-1}Z_lZ_l^T\Big|X^n$. Define the set 
$$G_n=\left\{||\Omega^{*1/2}\hat{\Sigma}\Omega^{*1/2}-I||\leq C\sqrt{\frac{p}{n}}\right\},$$
and we have $P_{\Sigma^*}^n(G_n^c)\leq\exp\big(-cp\big)$ by Lemma \ref{lem:wishartconcen}, for some $c,C>0$. The event $G_n$ implies $||\hat{\Sigma}-\Sigma^*||\leq C||\Sigma^*||\sqrt{\frac{p}{n}}$, by which we can deduce
\begin{eqnarray*}
 \left\|\Big(\hat{\Sigma}+\frac{1}{n}I\Big)^{-1}\right\| &=& \frac{1}{\lambda_{\min}\Big(\hat{\Sigma}+\frac{1}{n}I\Big)} \\
&\leq& \frac{1}{\lambda_{\min}(\Sigma^*)-\frac{1}{n}-||\hat{\Sigma}-\Sigma^*||} \\
&\leq& 2||\Omega^*||.
\end{eqnarray*}
Using the obtained results, we can bound the deviation of the sample covariance by
\begin{eqnarray*}
&& \left\|(n+p+b-1)(n\hat{\Sigma}+I)^{-1}-\Omega^*\right\| \\
&\leq& \left\|(n+p+b-1)(n\hat{\Sigma}+I)^{-1}\Bigg(\frac{n}{n+b+p-1}(\hat{\Sigma}-\Sigma^*)-\frac{b+p-1}{n+b+p-1}\Sigma^*+\frac{1}{n+p+b-1}I\Bigg)\Omega^*\right\| \\
&\leq&  \left\|\Big(\hat{\Sigma}+\frac{1}{n}I\Big)^{-1}\right\|||(\hat{\Sigma}-\Sigma^*)\Omega^*||+||(p+b-1)(n\hat{\Sigma}+I)^{-1}||+||(n\hat{\Sigma}+I)^{-1}\Omega^*|| \\
&\leq& 2C||\Sigma^*||^{1/2}||\Omega^*||^{3/2}\sqrt{\frac{p}{n}}+\frac{2(p+b-1)}{n}||\Omega^*||+\frac{2}{n}||\Omega^*||^2 \\
&\leq& C'||\Sigma^*||^{1/2}||\Omega^*||^{3/2}\sqrt{\frac{p}{n}},
\end{eqnarray*}
and the posterior deviation can be bounded by
\begin{eqnarray*}
&& P_{\Sigma^*}^n\Pi\Big(||\Omega-\Omega^*||>2C'||\Sigma^*||^{1/2}||\Omega^*||^{3/2}\sqrt{\frac{p}{n}}|X^n\Big) \\
&\leq& P_{\Sigma^*}^n\Pi\Big(||\Omega-\Omega^*||>2C'||\Sigma^*||^{1/2}||\Omega^*||^{3/2}\sqrt{\frac{p}{n}}|X^n\Big)\mathbb{I}_{G_n}+P_{\Sigma^*}^n(G_n^c) \\
&\leq& P_{\Sigma^*}^n\Pi\Big(\left\|\Omega-(n+p+b-1)(n\hat{\Sigma}+I)^{-1}\right\|>C'||\Sigma^*||^{1/2}||\Omega^*||^{3/2}\sqrt{\frac{p}{n}}|X^n\Big)\mathbb{I}_{G_n}+P_{\Sigma^*}^n(G_n^c) \\
&=& P_{\Sigma^*}^n\mathbb{P}\Bigg(\left\|\sum_{l=1}^{n+p+b-1}Z_lZ_l^T-(n+p+b-1)(n\hat{\Sigma}+I)^{-1}\right\|>C'||\Sigma^*||^{1/2}||\Omega^*||^{3/2}\sqrt{\frac{p}{n}}\Big|X^n\Bigg)\\
&& +P_{\Sigma^*}^n(G_n^c) \\
&\leq& P_{\Sigma^*}^n\mathbb{P}\Bigg(\left\|(n+p+b-1)^{-1}\sum_{l=1}^{n+p+b-1}W_lW_l^T-I\right\|>\frac{1}{2}\frac{1}{||(\hat{\Sigma}+n^{-1}I)^{-1}||}C'||\Sigma^*||^{1/2}||\Omega^*||^{3/2}\sqrt{\frac{p}{n}}\Big|X^n\Bigg)\\
&& +P_{\Sigma^*}^n(G_n^c) \\
&\leq& \mathbb{P}\Bigg(\left\|(n+p+b-1)^{-1}\sum_{l=1}^{n+p+b-1}W_lW_l^T-I\right\|>\frac{1}{4}C'||\Sigma^*||^{1/2}||\Omega^*||^{1/2}\sqrt{\frac{p}{n}}\Bigg) +P_{\Sigma^*}^n(G_n^c) \\
&\leq& \exp\big(-c'p\big),
\end{eqnarray*}
where we use $W_l\sim N(0,I)$ in the above equations.
In summary, we have proved
\begin{equation}
P_{\Sigma^*}^n\Pi\Bigg(||\Omega-\Omega^*||\leq 2C'||\Sigma^*||^{1/2}||\Omega^*||^{3/2}\sqrt{\frac{p}{n}}\Big|X^n\Bigg)\rightarrow 1, \label{eq:wishartpost}
\end{equation}
which implies
$$P_{\Sigma^*}^n\Pi\Big(||\Sigma-\Sigma^*||\leq M\sqrt{\frac{p}{n}}\Big|X^n\Big)\rightarrow 1,$$
with some sufficiently large $M>0$.
We choose
$$A_n=\left\{||\Sigma-\Sigma^*||\leq M\sqrt{\frac{p}{n}}\right\},$$
so that $\Pi(A_n|X^n)=1-o_P(1)$ is true.
For Theorem \ref{thm:main1}, let $\delta_n=M\sqrt{\frac{p}{n}}$. Then $\delta_n=o(1)$ by assumption,
and $A_n\subset\{||\Sigma-\Sigma^*||\leq \delta_n\}$.
For Theorem \ref{thm:main2}, let $\delta_n=M\sqrt{\frac{r^2p}{n}}$, then we have $\delta_n=o(1)$ and $A_n\subset\{\sqrt{rp}||\Sigma-\Sigma^*||\leq \delta_n\}$.

\noindent\textbf{Part II. }
Note that the proof for this part is the same for both Theorem \ref{thm:main1} and Theorem \ref{thm:main2} by letting $\Phi=-\Omega^*\Psi\Omega^*$. We introduce the notation
$$\tilde{\Phi}=\Big(\sqrt{2}\left\|\Sigma^{*1/2}\Phi\Sigma^{*1/2}\right\|_F\Big)^{-1}\Phi.$$
Now we study the integral $\int_{A_n}\exp\Big(l_n(\Omega_t)\Big)d\Pi(\Omega)$. Let $\mathcal{N}(p,b)$ be the normalizing constant of $\mathcal{W}_p(I,p+b-1)$. We have
\begin{eqnarray*}
&& \int_{A_n}\exp\Big(l_n(\Omega_t)\Big)d\Pi(\Omega) \\
&=& \mathcal{N}^{-1}(p,b)\int_{A_n}\exp\Big(l_n(\Omega+2tn^{-1/2}\tilde{\Phi})+\frac{b-2}{2}\log\det(\Omega)-\frac{1}{2}\text{tr}(\Omega)\Big)d\Omega \\
&=& \mathcal{N}^{-1}(p,b)\int_{A_n+2tn^{-1/2}\tilde{\Phi}}\exp\Big(l_n(\Gamma)+\frac{b-2}{2}\log\det(\Gamma-2tn^{-1/2}\tilde{\Phi})-\frac{1}{2}\text{tr}(\Gamma-2tn^{-1/2}\tilde{\Phi})\Big)d\Gamma \\
&=& \int_{A_n+2tn^{-1/2}\tilde{\Phi}}\exp\Big(l_n(\Omega)\Big)\exp\Big(\frac{b-2}{2}\log\det(I-2tn^{-1/2}\Omega^{-1}\tilde{\Phi})+\frac{1}{2}\text{tr}(2tn^{-1/2}\tilde{\Phi})\Big)d\Pi(\Omega).
\end{eqnarray*}
The above integrals are meaningful because $A_n\cup\big(A_n+2tn^{-1/2}\tilde{\Phi}\big)\subset\{\Omega:\Omega>0, \Omega=\Omega^T\}$. Note that
$$A_n+2tn^{-1/2}\tilde{\Phi}=\left\{\left\|(\Omega-2tn^{-1/2}\tilde{\Phi})^{-1}-\Sigma^*\right\|\leq M\sqrt{\frac{p}{n}}\right\}.$$
Since $||2tn^{-1/2}\tilde{\Phi}||= o\Big(\sqrt{\frac{p}{n}}\Big)$, there exist $M',M''$ arbitrarily close to $M$ such that $M'<M<M''$ and
$$A_n'\subset A_n+2tn^{-1/2}\tilde{\Phi}\subset A_n''$$
for $A_n'=\left\{||\Sigma-\Sigma^*||\leq M'\sqrt{\frac{p}{n}}\right\}$ and $A_n''=\left\{||\Sigma-\Sigma^*||\leq M''\sqrt{\frac{p}{n}}\right\}$. The result (\ref{eq:wishartpost}) implies  $\Pi(A_n'|X^n)=1-o_P(1)$ and $\Pi(A_n''|X^n)=1-o_P(1)$ are also true when $M',M,M''$ are large enough. Let $||\tilde{\Phi}||_N$ be the nuclear norm of $\tilde{\Phi}$, defined as the sum of its absolute eigenvalues. Note that on $A_n''$,
$$||\tilde{\Phi}||_N\leq C\frac{||\Phi||_N}{||\Phi||_F}\leq C\sqrt{p}.$$
Since
\begin{eqnarray*}
&& \sup_{A_n''}\left|\frac{b-2}{2}\log\det(I-2tn^{-1/2}\Omega^{-1}\tilde{\Phi})+\frac{1}{2}\text{tr}(2tn^{-1/2}\tilde{\Phi})\right| \\
&\leq& tn^{-1/2}\sup_{A_n''}\left||b-2|||\Omega^{-1/2}\tilde{\Phi}\Omega^{-1/2}||_N+||\tilde{\Phi}||_N\right| \\
&\leq& O(\sqrt{p/n}) \\
&=& o(1),
\end{eqnarray*}
we have
$$\int_{A_n}\exp\Big(l_n(\Omega_t)\Big)d\Pi(\Omega)\leq \big(1+o(1)\big)\int_{A_n''}\exp\Big(l_n(\Omega)\Big)d\Pi(\Omega),$$
and
$$\int_{A_n}\exp\Big(l_n(\Omega_t)\Big)d\Pi(\Omega)\leq \big(1-o(1)\big)\int_{A_n'}\exp\Big(l_n(\Omega)\Big)d\Pi(\Omega).$$
The facts that $\Pi(A_n'|X^n)=1-o_P(1)$ and $\Pi(A_n''|X^n)=1-o_P(1)$ lead to
$$\frac{\int_{A_n}\exp\Big(l_n(\Omega_t)\Big)d\Pi(\Omega)}{\int_{A_n}\exp\Big(l_n(\Omega)\Big)d\Pi(\Omega)}=1+o_P(1).$$
\end{proof}

\subsection{Proof of Lemma \ref{lem:gaussprior}} \label{sec:proofgauss}

Now we are going to prove Lemma \ref{lem:gaussprior}. Like the proof of Lemma \ref{lem:wishart}, it has two parts. The first part is to show posterior contraction on some appropriate set $A_n$. Note that Wishart prior is a conjugate prior. The posterior contraction can be directly calculated. For the Gaussian prior, its non-conjugacy requires to apply some general result from nonparametric Bayes theory. To be specific, we follow the testing approach in \cite{barron99} and \cite{ghosal00}. The outline of using testing approach to prove posterior contraction for Bayesian matrix estimation is referred to Section 5 in \cite{gao13}.

We first state some lemmas.
\begin{lemma} \label{lem:prior}
Assume $p^2=o(n/\log n)$ and $||\Sigma^*||\vee||\Omega^*||\leq\Lambda= O(1)$. For the Gaussian prior $\Pi$, we have
$$\Pi\Big(||\Omega||^2||\Sigma-\Sigma^*||_F^2\leq\frac{p^2\log n}{n}\Big)\geq\exp\Big(-Cp^2\log n\Big),$$
for some constant $C>0$.
\end{lemma}

The next lemma is  Lemma 5.1 in \cite{gao13}.
\begin{lemma} \label{lem:KL}
Let $K_n=\left\{||\Omega||^2||\Sigma-\Sigma^*||_F^2\leq\frac{p^2\log n}{n}\right\}$. Then for any $b>0$, we have
\begin{eqnarray*}
&& P_{\Sigma^*}^n\Bigg(\int \exp\Big(l_n(\Omega)-l_n(\Omega^*)\Big)d\Pi(\Omega)\leq \Pi(K_n)\exp\big(-(b+1)p^2\log n\big)\Bigg) \\
&\leq& \exp\Big(-Cb^2p^2\log n\Big),
\end{eqnarray*}
for some constant $C>0$.
\end{lemma}

The next lemma is Lemma 5.9 in \cite{gao13}.
\begin{lemma} \label{lem:test}
For $||\Sigma^*||\vee||\Omega^*||\leq\Lambda= O(1)$ and $||\Sigma_1||\vee||\Omega_1||\leq 2\Lambda$, there exist small $\delta,\delta'>0$ only depending on $\Lambda$, and a testing function $\phi$ such that
$$P_{\Sigma^*}^n\phi\leq 2\exp\Big(-C\delta'||\Sigma^*-\Sigma_1||_F^2\Big),$$
$$\sup_{\{\Sigma\in\text{supp}(\Pi):||\Sigma-\Sigma_1||_F\leq\delta ||\Sigma^*-\Sigma_1||_F\}}P_{\Sigma}^n(1-\phi)\leq 2\exp\Big(-C\delta'||\Sigma^*-\Sigma_1||_F^2\Big),$$
for some constant $C>0$.
\end{lemma}

\begin{proof}[Proof of Lemma \ref{lem:gaussprior}]
Like what we have done in the Wishart case, the proof has two parts. In the first part, we establish the first condition of the two theorems by proving a posterior contraction rate. In the second part, we establish the second condition of the two theorems by showing that a change of variable is negligible under Gaussian density.

\noindent\textbf{Part I. }  Define
$$A_n=\left\{||\Sigma-\Sigma^*||_F\leq M\sqrt{\frac{p^2\log n}{n}}\right\},$$
for some $M$ sufficiently large. Then, we may write
$$\Pi(A_n^c|X^n)=\frac{\int_{A_n^c}\exp\Big(l_n(\Omega)-l_n(\Omega^*)\Big)d\Pi(\Omega)}{\int\exp\Big(l_n(\Omega)-l_n(\Omega^*)\Big)d\Pi(\Omega)}=\frac{N_n}{D_n}.$$
Let us establish a testing between the following hypotheses:
$$H_0: \Sigma=\Sigma^*\quad\text{vs}\quad H_1: \Sigma\in A_n^c\cap\text{supp}(\Pi).$$
There exists $\{\Sigma_j\}_{j=1}^N\subset A_n^c\cap\text{supp}(\Pi)$, such that
$$A_n^c\cap\text{supp}(\Pi)\subset \text{supp}(\Pi)\cap\Bigg(\cup_{j=1}^N\left\{||\Sigma-\Sigma_j||_F\leq \sqrt{\frac{p^2\log n}{n}}\right\}\Bigg).$$
We choose the smallest $N$, which is determined by the covering number. Since $A_n^c\cap\text{supp}(\Pi)\subset\{||\Omega||_F\leq 2
\Lambda \sqrt{p}\}$, we have
$$\log N\leq C'p^2\log\Bigg(\frac{2\Lambda\sqrt{n}}{\sqrt{p\log n}}\Bigg)\leq Cp^2\log n.$$
By Lemma \ref{lem:test}, there exists $\phi_j$ such that
$$P_{\Sigma^*}^n\phi_j\leq 2\exp\Big(-CM^2p^2\log n\Big),$$
$$\sup_{\{\Sigma\in\text{supp}(\Pi):||\Sigma-\Sigma_j||_F\leq \sqrt{p^2\log n/n}\}}P_{\Sigma}^n(1-\phi_j)\leq 2\exp\Big(-CM^2p^2\log n\Big).$$
Define $\phi=\max_{1\leq j\leq N}\phi_j$. Using union bound to control the testing error, we have
$$P_{\Sigma^*}^n\phi\leq \exp\Big(-C_1M^2p^2\log n\Big),$$
$$\sup_{\{\Sigma\in A_n^c\cap\text{supp}(\Pi)\}}P_{\Sigma}^n(1-\phi)\leq\exp\Big(-C_1M^2p^2\log n\Big),$$
for sufficiently large $M$. We bound $\Pi(A_n^c|X^n)$ by
\begin{eqnarray*}
P_{\Sigma^*}^n\Pi(A_n^c|X^n) &\leq& P_{\Sigma^*}^n\Pi(A_n^c|X^n)(1-\phi)\mathbb{I}\left\{D_n>\exp(-2p^2\log n)\right\}\\
&& + P_{\Sigma^*}^n\phi+P_{\Sigma^*}^n\Big(D_n\leq \exp(-2p^2\log n)\Big) \\
&\leq& \exp\big(2p^2\log n\big)P_{\Sigma^*}^n\int_{A_n^c}\exp\Big(l_n(\Omega)-l_n(\Omega^*)\Big)(1-\phi)d\Pi(\Omega) \\
&& + P_{\Sigma^*}^n\phi+P_{\Sigma^*}^n\Big(D_n\leq \exp(-2p^2\log n)\Big) \\
&\leq& \exp\big(2p^2\log n\big)\int_{A_n^c}P_{\Sigma}^n(1-\phi)d\Pi(\Omega) \\
&& + P_{\Sigma^*}^n\phi+P_{\Sigma^*}^n\Big(D_n\leq \exp(-2p^2\log n)\Big) \\
&\leq& \exp\big(2p^2\log n\big)\sup_{\Sigma\in {A_n^c}\cap\text{supp}(\Pi)}P_{\Sigma}^n(1-\phi) \\
&& + P_{\Sigma^*}^n\phi+P_{\Sigma^*}^n\Big(D_n\leq \exp(-2p^2\log n)\Big).
\end{eqnarray*}
In the upper bound above, the first two terms are bounded by the testing error we have established. The last term can be bounded by combining the results of Lemma \ref{lem:prior} and Lemma \ref{lem:KL}. Hence, we have proved that
$$\Pi(A_n^c|X^n)=1-o_P(1).$$
For Theorem \ref{thm:main1}, let $\delta_n=M\sqrt{\frac{p^2\log n}{n}}$. Then $\delta_n=o(1)$ by assumption,
and $A_n\subset\{||\Sigma-\Sigma^*||\leq \delta_n\}$.
For Theorem \ref{thm:main2}, let $\delta_n=M\sqrt{\frac{rp^3\log n}{n}}$, then we have $\delta_n=o(1)$ and $A_n\subset\{\sqrt{rp}||\Sigma-\Sigma^*||\leq \delta_n\}$.

\noindent\textbf{Part II. }
Let $\Pi_G$ induce a prior distribution on symmetric $\Omega$ with each of the upper triangular element independently following $N(0,1)$. The density of $\Pi_G$ is
$$\frac{d\Pi_G(\Omega)}{d\Omega}=\xi_p^{-1}\exp\Big(-\frac{1}{2}||\bar{\Omega}||_F^2\Big),$$
where we use $\bar{\Omega}$ to zero out the lower triangular elements of $\Omega$ except the diagonal part and $\xi_p$ is the normalizing constant. Write
\begin{eqnarray*}
&& \int_{A_n}\exp\Big(l_n(\Omega_t)\Big)d\Pi(\Omega) \\
&=& \xi_p^{-1}\int_{A_n}\exp\Big(l_n(\Omega_t)-\frac{1}{2}||\bar{\Omega}||_F^2\Big)d\bar{\Omega}.
\end{eqnarray*}
Remembering the notation $\tilde{\Phi}$ defined in the proof of Lemma \ref{lem:wishart}, we have
\begin{eqnarray*}
&& \int_{A_n}\exp\Big(l_n(\Omega_t)-\frac{1}{2}||\bar{\Omega}||_F^2\Big)d\bar{\Omega} \\
&=& \int_{A_n+2tn^{-1/2}\tilde{\Phi}}\exp\Big(l_n(\Gamma)-\frac{1}{2}||\bar{\Gamma}-2tn^{-1/2}\bar{\tilde{\Phi}}||_F^2\Big)d\bar{\Gamma} \\
&=& \int_{A_n+2tn^{-1/2}\tilde{\Phi}}\exp\Big(l_n(\Gamma)-\frac{1}{2}||\bar{\Gamma}||_F^2+2tn^{-1/2}\text{tr}\big(\bar{\Gamma}\bar{\tilde{\Phi}}\big)-2t^2n^{-1}||\bar{\tilde{\Phi}}||_F^2\Big)d\bar{\Gamma}.
\end{eqnarray*}
We may choose $M',M''$ arbitrarily close to $M$ such that $M'<M<M''$ and $A_n'\subset A_n+2tn^{-1/2}\tilde{\Phi}\subset A_n''$ for
$$A_n'=\left\{||\Sigma-\Sigma^*||_F\leq M'\sqrt{\frac{p^2\log n}{n}}\right\},\quad A_n''=\left\{||\Sigma-\Sigma^*||_F\leq M''\sqrt{\frac{p^2\log n}{n}}\right\}.$$
This can always be done because $||2tn^{-1/2}\tilde{\Phi}||_F= O\big(n^{-1/2}\big)=o\Bigg(\sqrt{\frac{p^2\log n}{n}}\Bigg)$. Moreover, we have $\Pi(A_n'|X^n)=1-o_P(1)$, $\Pi(A_n''|X^n)=1-o_P(1)$ and
\begin{eqnarray*}
&& \sup_{A_n''}\left|2tn^{-1/2}\text{tr}\big(\bar{\Gamma}\bar{\tilde{\Phi}}\big)-2t^2n^{-1}||\bar{\tilde{\Phi}}||_F^2\right| \\
&\leq& C\sup_{A_n''}\left|n^{-1/2}||\Gamma||_F||\tilde{\Phi}||_F+n^{-1}||\tilde{\Phi}||_F^2\right| \\
&=& O\Bigg(\sqrt{\frac{p}{n}}+\frac{1}{n}\Bigg) \\
&=& o(1).
\end{eqnarray*}
Therefore, using the same argument in the proof of Lemma \ref{lem:wishart}, we have
$$\frac{\int_{A_n}\exp\Big(l_n(\Omega_t)\Big)d\Pi(\Omega)}{\int_{A_n}\exp\Big(l_n(\Omega)\Big)d\Pi(\Omega)}=1+o_P(1).$$
This completes the proof.
\end{proof}

\subsection{Proof of Technical Lemmas}

\begin{proof}[Proof of Lemma \ref{lem:likexpansion}]
First, we show $\Omega_t$ is a valid precision matrix under the event $A_n$, i.e., $\Omega_t>0$. 
Using Weyl's theorem, we have
$$|\lambda_{\min}(\Omega_t)-\lambda_{\min}(\Omega^*)| \leq ||\Omega_t-\Omega|| + ||\Omega-\Omega^*||,$$
where the first term is bounded by
$$||\Omega_t-\Omega||\leq \frac{\sqrt{2}t}{\sqrt{n}}\frac{||\Phi||}{\left\|\Sigma^{*1/2}\Phi\Sigma^{*1/2}\right\|_F}= O(n^{-1/2}).$$
Hence,
$$|\lambda_{\min}(\Omega_t)-\lambda_{\min}(\Omega^*)| \leq O(n^{-1/2}) + ||\Omega-\Omega^*||.$$
Under the current assumption, $O(n^{-1/2}) + ||\Omega-\Omega^*||=o\Big(\lambda_{\min}(\Omega^*)\Big)$. Hence, $\lambda_{\min}(\Omega_t)>0$.
Knowing the fact that $l_n(\Omega_t)$ is well-defined, we study $l_n(\Omega_t)-l_n(\Omega)$,
\begin{eqnarray*}
l_n(\Omega_t)-l_n(\Omega) &=& \frac{n}{2}\text{tr}\Big(\hat{\Sigma}(\Omega-\Omega_t)\Big)+\frac{n}{2}\log\det\Big(I-(\Omega-\Omega_t)\Sigma\Big) \\
&=& \frac{n}{2}\text{tr}\Big((\hat{\Sigma}-\Sigma)(\Omega-\Omega_t)\Big)+\frac{n}{2}\text{tr}\Big(\Sigma^{1/2}(\Omega-\Omega_t)\Sigma^{1/2}\Big) \\
&& +\frac{n}{2}\log\det\Big(I-\Sigma^{1/2}(\Omega-\Omega_t)\Sigma^{1/2}\Big).
\end{eqnarray*}
Let $\{h_j\}_{j=1}^p$ be eigenvalues of $\Sigma^{1/2}(\Omega-\Omega_t)\Sigma^{1/2}$. Then, we have
\begin{eqnarray*}
&& \frac{n}{2}\text{tr}\Big(\Sigma^{1/2}(\Omega-\Omega_t)\Sigma^{1/2}\Big)+\frac{n}{2}\log\det\Big(I-\Sigma^{1/2}(\Omega-\Omega_t)\Sigma^{1/2}\Big) \\
&=& \frac{n}{2}\sum_{j=1}^p \Big(h_j+\log(1-h_j)\Big) \\
&=& -\frac{n}{4}\sum_{j=1}^p h_j^2 -\frac{n}{2}\sum_{j=1}^p\int_0^{h_j}\frac{(h_j-s)^2}{(1-s)^3}ds \\
&=& -\frac{n}{4}||\Sigma^{1/2}(\Omega-\Omega_t)\Sigma^{1/2}||_F^2-\frac{n}{2}\sum_{j=1}^p\int_0^{h_j}\frac{(h_j-s)^2}{(1-s)^3}ds,
\end{eqnarray*}
where $\int_0^{h_j}\frac{(h_j-s)^2}{(1-s)^3}ds$ is the remainder of the Taylor expansion.
Therefore, we have obtained the expansion
\begin{eqnarray*}
&& l_n(\Omega_t)-l_n(\Omega) \\
&=& \frac{n}{2}\text{tr}\Big((\hat{\Sigma}-\Sigma)(\Omega-\Omega_t)\Big)-\frac{n}{4}||\Sigma^{1/2}(\Omega-\Omega_t)\Sigma^{1/2}||_F^2-\frac{n}{2}\sum_{j=1}^p\int_0^{h_j}\frac{(h_j-s)^2}{(1-s)^3}ds \\
&=& \frac{t\sqrt{n}}{\sqrt{2}||\Sigma^{*1/2}\Phi\Sigma^{*1/2}||_F}\text{tr}\Big((\Sigma-\hat{\Sigma})\Phi\Big)-\frac{t^2}{2}\frac{||\Sigma^{1/2}\Phi\Sigma^{1/2}||_F^2}{||\Sigma^{*1/2}\Phi\Sigma^{*1/2}||_F^2}-\frac{n}{2}\sum_{j=1}^p\int_0^{h_j}\frac{(h_j-s)^2}{(1-s)^3}ds.
\end{eqnarray*}
The proof is complete.
\end{proof}

\begin{proof} [Proof of Lemma \ref{lem:prior}]
Define $\Pi_G$ to be the distribution which specifies i.i.d. $N(0,1)$ on the upper triangular part of $\Omega$ and then take the lower triangular part to satisfy $\Omega^T=\Omega$. Define
$$D=\{||\Omega||\vee||\Sigma||\leq 2\Lambda\}.$$
Then according to the definition of $\Pi$, we have
$$\Pi(B)=\frac{\Pi_G(B\cap D)}{\Pi_G(D)},\quad\text{for any }B.$$
Since $\Pi_G(D)\leq 1$, we have
$$\Pi(B)\geq\Pi_G(B\cap D),\quad\text{for any }B.$$
In particular, we have
$$\Pi\Big(||\Omega||^2||\Sigma-\Sigma^*||_F^2\leq p^2\log n/n\Big)\geq\Pi_G\Big(||\Omega||^2||\Sigma-\Sigma^*||_F^2\leq p^2\log n/n, ||\Omega||\vee||\Sigma||\leq 2\Lambda\Big).$$
Since $p^2/n=o(1)$, we have
$$\left\{||\Omega-\Omega^*||_F\leq \frac{p\sqrt{\log n}}{(2\Lambda)^3\sqrt{n}}\right\}\subset \left\{||\Omega||^2||\Sigma-\Sigma^*||_F^2\leq \frac{p^2\log n}{n}, ||\Omega||\vee||\Sigma||\leq 2\Lambda\right\}.$$
Thus,
$$\Pi\Big(||\Omega||^2||\Sigma-\Sigma^*||_F^2\leq p^2\log n/n\Big)\geq\Pi_G\Big(||\Omega-\Omega^*||_F\leq \frac{p\sqrt{\log n}}{(2\Lambda)^3\sqrt{n}}\Big).$$
Calculate using Gaussian density directly, for example, according to Lemma E.1 in \cite{gao13}, and we have
\begin{eqnarray*}
&&  \Pi_G\Big(||\Omega-\Omega^*||_F\leq \frac{p\sqrt{\log n}}{(2\Lambda)^3\sqrt{n}}\Big) \\
&\geq& e^{-||\Omega^*||_F^2}\left(\mathbb{P}\Big(|Z|^2\leq \frac{\log n}{c n}\Big)\right)^{p(p+1)/2} \\
&\geq& \exp\Big(-||\Omega^*||_F^2-Cp^2\log n\Big),
\end{eqnarray*}
where $Z\sim N(0,1)$. The proof is complete by observing that $||\Omega^*||_F^2= o(p^2\log n)$ under the assumption.
\end{proof}

\appendix

\section{Proof of Theorem \ref{thm:LDA} \& Theorem \ref{thm:QDA}}

\begin{lemma} \label{lem:LDAexpansion}
Under the setting of Theorem \ref{thm:LDA},
assume $p^2/n=o(1)$ and $||\Sigma^*||\vee||\Omega^*||= O(1)$, then we have
\begin{eqnarray*}
&& l_n(\mu_{X,t},\mu_{Y,t},\Omega_t)-l_n(\mu_X,\mu_Y,\Omega) \\
&=& \frac{2t\sqrt{n}}{V}\text{tr}\Big((\Sigma-\hat{\Sigma})\Phi\Big)-\frac{t\sqrt{n}}{V}(\bar{X}-\mu_X)^T\Omega^*\xi_X-\frac{t\sqrt{n}}{V}(\bar{Y}-\mu_Y)^T\Omega^*\xi_Y-\frac{1}{2}t^2+o_P(1),
\end{eqnarray*}
uniformly on $A_n$.
\end{lemma}
\begin{proof}
Since
\begin{eqnarray*}
&& l_n(\mu_{X,t},\mu_{Y,t},\Omega_t)-l_n(\mu_X,\mu_Y,\Omega) \\
&=& \Big(l_X(\mu_{X,t},\Omega_t)-l_X(\mu_X,\Omega)\Big) + \Big(l_Y(\mu_{Y,t},\Omega_t)-l_Y(\mu_Y,\Omega)\Big),
\end{eqnarray*}
we expand both quantities in the brackets using the general notation $l(\mu_t,\Omega_t)-l(\mu,\Omega)$. Using Taylor expansion as in the proof of Lemma \ref{lem:likexpansion} and the notation $\tilde{\Sigma}=\frac{1}{n}\sum_{i=1}^n(X_i-\mu)(X_i-\mu)^T$, we have
\begin{eqnarray*}
&& l(\mu_t,\Omega_t)-l(\mu,\Omega) \\
&=& \frac{n}{2}\text{tr}\Big((\tilde{\Sigma}-\Sigma)(\Omega-\Omega_t)\Big)-n(\mu-\mu_t)^T\Omega_t(\bar{X}-\mu)-\frac{n}{2}\left\|\Sigma^{1/2}(\Omega-\Omega_t)\Sigma^{1/2}\right\|_F^2\\
&& -\frac{n}{2}(\mu-\mu_t)^T\Omega_t(\mu-\mu_t)-\frac{n}{2}\sum_{j=1}^p\int_{0}^{h_j}\frac{(h_j-s)^2}{(1-s)^3}ds \\
&=& \frac{t\sqrt{n}}{V}\text{tr}\Big((\Sigma-\tilde{\Sigma})\Phi\Big)+\frac{t\sqrt{n}}{V}\xi^T\Omega_t(\bar{X}-\mu) \\
&& -\frac{t^2}{V^2}\left\|\Sigma^{1/2}\Phi\Sigma^{1/2}\right\|_F^2 - \frac{t^2}{2V^2}\xi^T\Omega_t\xi-\frac{n}{2}\sum_{j=1}^p\int_{0}^{h_j}\frac{(h_j-s)^2}{(1-s)^3}ds,
\end{eqnarray*}
where $\{h_j\}_{j=1}^p$ are eigenvalues of $\Sigma^{1/2}(\Omega-\Omega_t)\Sigma^{1/2}$. The same proof in Lemma \ref{lem:likexpansion} implies
$$\left|\frac{n}{2}\sum_{j=1}^p\int_{0}^{h_j}\frac{(h_j-s)^2}{(1-s)^3}ds\right|= o(1),\quad \text{on }A_n.$$
Therefore,
\begin{eqnarray*}
&& l_n(\mu_{X,t},\mu_{Y,t},\Omega_t)-l_n(\mu_X,\mu_Y,\Omega) \\
&=& \frac{2t\sqrt{n}}{V}\text{tr}\Bigg(\Big(\Sigma-\frac{1}{2}\big(\tilde{\Sigma}_X+\tilde{\Sigma}_Y\big)\Big)\Phi\Bigg)+\frac{t\sqrt{n}}{V}(\bar{X}-\mu_X)^T\Omega_t\xi_X+\frac{t\sqrt{n}}{V}(\bar{Y}-\mu_Y)^T\Omega_t\xi_Y \\
&& -\frac{t^2}{2V^2}\Big(4\left\|\Sigma^{1/2}\Phi\Sigma^{1/2}\right\|_F^2+\xi_X^T\Omega_t\xi_X+\xi_Y^T\Omega_t\xi_Y\Big) + o(1).
\end{eqnarray*}
We approximate $\frac{2t\sqrt{n}}{V}\text{tr}\Bigg(\Big(\Sigma-\frac{1}{2}\big(\tilde{\Sigma}_X+\tilde{\Sigma}_Y\big)\Big)\Phi\Bigg)$ by $\frac{2t\sqrt{n}}{V}\text{tr}\Big((\Sigma-\hat{\Sigma})\Phi\Big)$, and the approximation error is bounded by
\begin{eqnarray*}
&& \frac{C\sqrt{n}}{V}\left|\text{tr}\Big((\hat{\Sigma}_X-\tilde{\Sigma}_X)\Phi\Big)\right|+ \frac{C\sqrt{n}}{V}\left|\text{tr}\Big((\hat{\Sigma}_Y-\tilde{\Sigma}_Y)\Phi\Big)\right| \\
&\leq& CV^{-1}\sqrt{n}\Big((\bar{X}-\mu_X)^T\Phi(\bar{X}-\mu_X)+(\bar{Y}-\mu_Y)^T\Phi(\bar{Y}-\mu_Y)\Big) \\
&\leq& C||\Phi||V^{-1}\sqrt{n}\Big(||\bar{X}-\mu_X||^2+||\bar{Y}-\mu_Y||^2\Big) \\
&\leq& C||\Phi||V^{-1}\sqrt{n}\Big(||\mu_X-\mu_X^*||^2+||\mu_Y-\mu_Y^*||^2+O_P(p/n)\Big) \\
&=& o_P(1),
\end{eqnarray*}
under $A_n$ and the assumption $p^2/n=o(1)$, where we have used the fact that $||\Phi||/V\leq C$. We approximate $\frac{t\sqrt{n}}{V}(\bar{X}-\mu_X)^T\Omega_t\xi_X$ by $\frac{t\sqrt{n}}{V}(\bar{X}-\mu_X)^T\Omega^*\xi_X$, and the difference is bounded by
\begin{eqnarray*}
&& C\sqrt{n}V^{-1}\left|\xi_X^T(\Omega_t-\Omega)(\bar{X}-\mu_X)\right| + C\sqrt{n}V^{-1}\left|\xi_X^T(\Omega-\Omega^*)(\bar{X}-\mu_X)\right| \\
&\leq& CV^{-2}||\xi_X||||\bar{X}-\mu_X||||\Phi|| + C\sqrt{n}V^{-1}||\xi_X||||\Omega-\Omega^*||||\bar{X}-\mu_X|| \\
&\leq& C||\bar{X}-\mu_X|| + C\sqrt{n}||\Omega-\Omega^*||||\bar{X}-\mu_X|| \\
&\leq& C\Big(||\mu_X-\mu_X^*||+\sqrt{n}||\mu_X-\mu_X^*||||\Omega-\Omega^*||+\sqrt{p}||\Omega-\Omega^*||+O_P(\sqrt{p/n})\Big) \\
&=& o_P(1),
\end{eqnarray*}
under $A_n$ and the fact that $||\xi_X||/V\leq C$. Using the same argument, we can also approximate $\frac{t\sqrt{n}}{V}(\bar{Y}-\mu_Y)^T\Omega_t\xi_Y$ by $\frac{t\sqrt{n}}{V}(\bar{Y}-\mu_Y)^T\Omega^*\xi_Y$. Now we approximate the quadratic terms. Using the same argument in the proof of Lemma \ref{lem:likexpansion}, we have
$$\frac{\left|\left\|\Sigma^{1/2}\Phi\Sigma^{1/2}\right\|_F^2-\left\|\Sigma^{*1/2}\Phi\Sigma^{*1/2}\right\|_F^2\right|}{V^2}= o(1).$$
We also have
$$\frac{|\xi_X^T(\Omega_t-\Omega^*)\xi_X|}{V^2}\leq C\Big(||\Omega-\Omega^*||+||\Omega_t-\Omega||\Big)= o(1),$$
and the same bound for $\frac{|\xi_Y^T(\Omega_t-\Omega^*)\xi_Y|}{V^2}$. Therefore,
$$ \frac{t^2}{2V^2}\Big(4\left\|\Sigma^{1/2}\Phi\Sigma^{1/2}\right\|_F^2+\xi_X^T\Omega_t\xi_X+\xi_Y^T\Omega_t\xi_Y\Big) =\frac{t^2}{2}+o(1),$$
on $A_n$. The proof is complete by considering all the approximations above.
\end{proof}

\begin{lemma} \label{lem:LDAfunctional}
Under the same setting of Lemma \ref{lem:LDAexpansion} and further assume $V^{-1}= O(1)$, we have
\begin{eqnarray*}
&& \frac{t\sqrt{n}}{V}\Big(\Delta(\mu_X,\mu_Y,\Omega)-\Delta(\bar{X},\bar{Y},\hat{\Sigma}^{-1})\Big) \\
&=& \frac{2t\sqrt{n}}{V}\text{tr}\Big((\Sigma-\hat{\Sigma})\Phi\Big)-\frac{t\sqrt{n}}{V}(\bar{X}-\mu_X)^T\Omega^*\xi_X-\frac{t\sqrt{n}}{V}(\bar{Y}-\mu_Y)^T\Omega^*\xi_Y + o_P(1),
\end{eqnarray*}
uniformly on $A_n$.
\end{lemma}

\begin{proof}[Proof of Theorem \ref{thm:LDA}]
Combining Lemma \ref{lem:LDAexpansion} and Lemma \ref{lem:LDAfunctional}, we have
$$l_n(\mu_{X,t},\mu_{Y,t},\Omega_t)-l_n(\mu_X,\mu_Y,\Omega)=\frac{t\sqrt{n}}{V}\Big(\Delta(\mu_X,\mu_Y,\Omega)-\Delta(\bar{X},\bar{Y},\hat{\Sigma}^{-1})\Big)-\frac{1}{2}t^2+o_P(1),$$
uniformly in $A_n$. The remaining of the proof is the same as the proof of Theorem \ref{thm:main1}.
\end{proof}

The proof of Theorem \ref{thm:QDA},  is very similar to the proof of Theorem \ref{thm:LDA}. We simply state the technical steps in the following lemmas and omit the details of the proof.

\begin{lemma} \label{lem:QDAexpansion}
Under the setting of Theorem \ref{thm:QDA},
assume $p^2/n=o(1)$ and $||\Sigma^*||\vee||\Omega^*||= O(1)$, then we have
\begin{eqnarray*}
&& l_n(\mu_{X,t},\mu_{Y,t},\Omega_{X,t},\Omega_{Y,t})-l_n(\mu_X,\mu_Y,\Omega_X,\Omega_Y) \\
&=& \frac{t\sqrt{n}}{V}\text{tr}\Big((\Sigma_X-\hat{\Sigma}_X)\Phi_X\Big)+\frac{t\sqrt{n}}{V}\text{tr}\Big((\Sigma_Y-\hat{\Sigma}_Y)\Phi_Y\Big) \\
&& -\frac{t\sqrt{n}}{V}(\bar{X}-\mu_X)^T\Omega^*\xi_X-\frac{t\sqrt{n}}{V}(\bar{Y}-\mu_Y)^T\Omega^*\xi_Y-\frac{1}{2}t^2+o_P(1),
\end{eqnarray*}
uniformly on $A_n$.
\end{lemma}

\begin{lemma} \label{lem:QDAfunctional}
Under the same setting of Lemma \ref{lem:QDAexpansion} and further assume $V^{-1}= O(1)$ and $p^3/n=o(1)$, then
\begin{eqnarray*}
&& \frac{t\sqrt{n}}{V}\Big(\Delta(\mu_X,\mu_Y,\Omega_X,\Omega_Y)-\Delta(\bar{X},\bar{Y},\hat{\Sigma}_X^{-1},\hat{\Sigma}_Y^{-1})\Big) \\
&=& \frac{t\sqrt{n}}{V}\text{tr}\Big((\Sigma_X-\hat{\Sigma}_X)\Phi_X\Big)+\frac{t\sqrt{n}}{V}\text{tr}\Big((\Sigma_Y-\hat{\Sigma}_Y)\Phi_Y\Big) \\
&& -\frac{t\sqrt{n}}{V}(\bar{X}-\mu_X)^T\Omega^*\xi_X-\frac{t\sqrt{n}}{V}(\bar{Y}-\mu_Y)^T\Omega^*\xi_Y+ o_P(1),
\end{eqnarray*}
uniformly on $A_n$.
\end{lemma}

\begin{proof}[Proof of Theorem \ref{thm:QDA}]
Combining Lemma \ref{lem:QDAexpansion} and Lemma \ref{lem:QDAfunctional}, we have
\begin{eqnarray*}
&& l_n(\mu_{X,t},\mu_{Y,t},\Omega_{X,t},\Omega_{Y,t})-l_n(\mu_X,\mu_Y,\Omega_X,\Omega_Y) \\
&=& \frac{t\sqrt{n}}{V}\Big(\Delta(\mu_X,\mu_Y,\Omega_X,\Omega_Y)-\Delta(\bar{X},\bar{Y},\hat{\Sigma}_X^{-1},\hat{\Sigma}_Y^{-1})\Big)-\frac{1}{2}t^2+o_P(1),
\end{eqnarray*}
uniformly in $A_n$. The remaining of the proof is the same as the proof of Theorem \ref{thm:main1}.
\end{proof}

\section{Proof of  Theorem \ref{thm:LDAgauss} \& Theorem \ref{thm:QDAgauss}}

In this section, we are going to prove Theorem \ref{thm:LDAgauss} and Theorem \ref{thm:QDAgauss}.
Due to the similarity of the two theorems, we only present the details of the proof of Theorem \ref{thm:QDAgauss}. The proof of Theorem \ref{thm:LDAgauss} will be outlined.
By the remark after Theorem \ref{thm:QDAgauss}, it is sufficient to check the two conditions in Theorem \ref{thm:QDAgauss} for $X$ and $Y$ separately. Therefore, we only prove for the $X$ part and omit the subscript $X$ from now on. 

Denote the prior for $(\Omega,\mu)$ as $\Pi=\Pi_{\Omega}\times\Pi_{\mu}$. The following lemma is a generalization of Lemma \ref{lem:KL} to the nonzero mean case.
\begin{lemma} \label{lem:KLDA}
Let $\epsilon$ be any sequence such that $\epsilon\rightarrow 0$. Define
$$K_n=\left\{||\Omega||^2||\Sigma-\Sigma^*||_F^2+2||\Omega||||\mu-\mu^*||^2\leq \epsilon^2\right\}.$$
Then for any $b>0$, we have
\begin{eqnarray*}
&& P_{\Sigma^*}^n\Bigg(\int \exp\Big(l_n(\Omega)-l_n(\Omega^*)\Big)d\Pi(\Omega)\leq \Pi(K_n)\exp\big(-(b+1)n\epsilon^2\big)\Bigg) \\
&\leq& \exp\Big(-Cb^2n\epsilon^2\Big),
\end{eqnarray*}
for some constant $C>0$.
\end{lemma}
\begin{proof}
We renormalize the prior $\Pi$ as $\tilde{\Pi}=\Pi(K_n)^{-1}\tilde{\Pi}$ so that $\tilde{\Pi}$ is a distribution with support within $K_n$. Write $\mathbb{E}_{\tilde{\Pi}}$ to be the expectation using probability $\tilde{\Pi}$. Define the random variable
$$Y_i=\int \log\frac{dP_{\Sigma}}{dP_{\Sigma^*}}(X_i)d\tilde{\Pi}(\Omega)=c+\frac{1}{2}(X_i-\mu^*)^T(\Omega^*-\mathbb{E}_{\tilde{\Pi}}\Omega)(X_i-\mu^*)+(X_i-\mu^*)^T\mathbb{E}_{\tilde{\Pi}}\Big(\Omega(\mu-\mu^*)\Big),$$
for $i=1,...,n$, where $c$ is a constant independent of $X_1,...,X_n$. Then, $Y_i$ is a sub-exponential random variable with mean
\begin{eqnarray*}
-P_{\Sigma^*}Y_i &=& \int D(P_{\Sigma^*}||P_{\Sigma})d\tilde{\Pi}(\Omega) \\
&\leq& \int \Bigg(\frac{1}{4}||\Omega||^2||\Sigma-\Sigma^*||_F^2+\frac{1}{2}||\Omega||||\mu-\mu^*||^2\Bigg)d\tilde{\Pi}(\Omega) \\
&\leq&\epsilon^2/4.
\end{eqnarray*}
Thus, by Jensen's inequality, we have
\begin{eqnarray*}
&& P_{\Sigma^*}^n\Bigg(\int\frac{dP_{\Sigma}^n}{dP_{\Sigma^*}^n}(X^n)d\tilde{\Pi}(\Omega)\leq\exp\Big(-(b+1)n\epsilon^2\Big)\Bigg) \\
&\leq& P_{\Sigma^*}^n\Bigg(\frac{1}{n}\sum_{i=1}^nY_i\leq -(b+1)\epsilon^2\Bigg) \\
&\leq& P_{\Sigma^*}^n\Bigg(\frac{1}{n}\sum_{i=1}^n(Y_i-P_{\Sigma^*}Y_i)\leq-b\epsilon^2\Bigg) \\
&=& P_{\Sigma}^n\Bigg(\frac{1}{n}\sum_{i=1}^n(Y_{1i}-P_{\Sigma^*}Y_{1i})+\frac{1}{n}\sum_{i=1}^n(Y_{2i}-P_{\Sigma^*}Y_{2i})\leq-b\epsilon^2\Bigg),
\end{eqnarray*}
where in the last equality we defined
$$Y_{1i}=\frac{1}{2}(X_i-\mu^*)^T(\Omega^*-\mathbb{E}_{\tilde{\Pi}}\Omega)(X_i-\mu^*),\quad Y_{2i}=(X_i-\mu^*)^T\mathbb{E}_{\tilde{\Pi}}\Big(\Omega(\mu-\mu^*)\Big),$$
for $i=1,...,n$. By union bound, we have
\begin{eqnarray*}
&& P_{\Sigma}^n\Bigg(\frac{1}{n}\sum_{i=1}^n(Y_{1i}-P_{\Sigma^*}Y_{1i})+\frac{1}{n}\sum_{i=1}^n(Y_{2i}-P_{\Sigma^*}Y_{2i})\leq-b\epsilon^2\Bigg) \\
&\leq& P_{\Sigma}^n\Bigg(\frac{1}{n}\sum_{i=1}^n(Y_{1i}-P_{\Sigma^*}Y_{1i})\leq-\frac{b\epsilon^2}{2}\Bigg) + P_{\Sigma}^n\Bigg(\frac{1}{n}\sum_{i=1}^n(Y_{2i}-P_{\Sigma^*}Y_{2i})\leq-\frac{b\epsilon^2}{2}\Bigg).
\end{eqnarray*}
In the proof of Lemma 5.1 of \cite{gao13}, we have shown that
$$P_{\Sigma}^n\Bigg(\frac{1}{n}\sum_{i=1}^n(Y_{1i}-P_{\Sigma^*}Y_{1i})\leq-\frac{b\epsilon^2}{2}\Bigg)\leq \exp\Big(-Cb^2n\epsilon^2\Big).$$
Hence, it is sufficient to bound the second term. Define $Z_i=\Omega^{*1/2}(X_i-\mu^*)$, and then we have
$$Y_{2i}-P_{\Sigma^*}Y_{2i}=Z_i^Ta\quad\text{and}\quad Z_i\sim N(0,I_{p\times p})$$
with $a=\Sigma^{*1/2}\mathbb{E}_{\tilde{\Pi}}\Big(\Omega(\mu-\mu^*)\Big)$.  By Bernstein's inequality (see, for example, Proposition 5.16 of \cite{vershynin10}), we have
$$\mathbb{P}\Bigg(\frac{1}{n}\sum_{i=1}^na^T Z_i\leq -\frac{b\epsilon^2}{2}\Bigg)\leq \exp\Bigg(-C\min\Big(\frac{(nb\epsilon^2)^2}{n||a||^2},\frac{nb\epsilon^2}{||a||_{\infty}}\Big)\Bigg).$$
Since
\begin{eqnarray*}
||a||^2 &\leq& ||\Sigma^*||\left\|\mathbb{E}_{\tilde{\Pi}}\Big(\Omega(\mu-\mu^*)\Big)\right\|^2 \\
&\leq& ||\Sigma^*||\mathbb{E}_{\tilde{\Pi}}||\Omega(\mu-\mu^*)||^2 \\
&\leq& C'\epsilon^2,
\end{eqnarray*}
and $||a||_{\infty}\leq ||a||\leq \sqrt{C'\epsilon^2}$, then
$$\mathbb{P}\Bigg(\frac{1}{n}\sum_{i=1}^na^T Z_i\leq -\frac{b\epsilon^2}{2}\Bigg)\leq \exp\Big(-C\min\big(b^2n\epsilon^2,bn\epsilon\big)\Big)=\exp\Big(-Cb^2n\epsilon^2\Big),$$
because $\epsilon\rightarrow 0$. The conclusion follows the fact that
\begin{eqnarray*}
&& P_{\Sigma^*}\Bigg(\int\frac{dP_{\Sigma}^n}{dP_{\Sigma^*}^n}(X^n)d\Pi(\Omega)\leq\Pi(K_n)\exp\Big(-(b+1)n\epsilon^2\Big)\Bigg) \\
&\leq& P_{\Sigma^*}^n\Bigg(\int\frac{dP_{\Sigma}^n}{dP_{\Sigma^*}^n}(X^n)d\tilde{\Pi}(\Omega)\leq\exp\Big(-(b+1)n\epsilon^2\Big)\Bigg).
\end{eqnarray*}
\end{proof}

The following lemma proves prior concentration.
\begin{lemma} \label{lem:priorDA}
Assume $p^2=o(n/\log n)$, $||\mu^*||= O(1)$ and $||\Sigma^*||\vee||\Omega^*||\leq\Lambda= O(1)$. For the prior $\Pi=\Pi_{\Omega}\times\Pi_{\mu}$, we have
$$\Pi\Big(||\Omega||^2||\Sigma-\Sigma^*||_F^2+2||\Omega||||\mu-\mu^*||^2\leq\frac{p^2\log n}{n}\Big)\geq\exp\Big(-Cp^2\log n\Big),$$
for some constant $C>0$.
\end{lemma}
\begin{proof}
We have
\begin{eqnarray*}
&& \Pi\Big(||\Omega||^2||\Sigma-\Sigma^*||_F^2+2||\Omega||||\mu-\mu^*||^2\leq\frac{p^2\log n}{n}\Big) \\
&\geq& \Pi\Big(||\Omega||^2||\Sigma-\Sigma^*||_F^2+4\Lambda||\mu-\mu^*||^2\leq\frac{p^2\log n}{n}\Big) \\
&\geq& \Pi\Big(||\Omega||^2||\Sigma-\Sigma^*||_F^2\leq\frac{p^2\log n}{2n}, 4\Lambda||\mu-\mu^*||^2\leq\frac{p^2\log n}{2n}\Big) \\
&=& \Pi_{\Omega}\Big(||\Omega||^2||\Sigma-\Sigma^*||_F^2\leq\frac{p^2\log n}{2n}\Big)\Pi_{\mu}\Big(4\Lambda||\mu-\mu^*||^2\leq\frac{p^2\log n}{2n}\Big),
\end{eqnarray*}
where the first term is lower bounded in Lemma \ref{lem:prior}.
It is sufficient to lower bound $\Pi_{\mu}\Big(4\Lambda||\mu-\mu^*||^2\leq\frac{p^2\log n}{2n}\Big)$. By the definition of  Gaussian density,
\begin{eqnarray*}
&& \Pi_{\mu}\Big(4\Lambda||\mu-\mu^*||^2\leq\frac{p^2\log n}{2n}\Big) \\
&\geq&e^{-||\mu^*||^2/2}\Bigg(\mathbb{P}\Big(|Z|^2\leq\frac{p\log n}{c n}\Big)\Bigg)^p \\
&\geq&\exp\Big(-||\mu^*||^2/2-Cp\log n\Big).
\end{eqnarray*}
The proof is complete by noticing $||\mu^*||= O(1)$.
\end{proof}

\begin{lemma} \label{lem:testmean}
Assume $||\Sigma^*||\vee||\Omega^*||\leq \Lambda= O(1)$. Then for any constant $M>0$, there exists a testing function $\phi$ such that
$$P_{(\mu^*,\Omega^*)}^n\phi\leq\exp\Big(-CM^2p^2\log n\Big),$$
$$\sup_{\{(\mu,\Omega)\in\text{supp}(\Pi):||\mu-\mu^*||>M\sqrt{\frac{p^2\log n}{n}}\}}P_{(\mu,\Omega)}^n(1-\phi)\leq \exp\Big(-CM^2p^2\log n\Big),$$
for some constant $C>0$.
\end{lemma}
\begin{proof}
Use notation $\epsilon^2=p^2\log n/n$.
Consider the testing function
$$\phi=\left\{||\bar{X}-\mu^*||>\frac{M\epsilon}{2}\right\}.$$
Then we have
$$P_{(\mu^*,\Omega^*)}^n\phi=\mathbb{P}\Bigg(\frac{1}{\sqrt{n}}||\Sigma^{*1/2}Z\Sigma^{*1/2}||>\frac{M\epsilon}{2}\Bigg)\leq \mathbb{P}\Big(||Z||^2\geq CM^2n\epsilon^2\Big),$$
where $Z\sim N(0,I_{p\times p})$. We also have for any $(\mu,\Omega)$ in the alternative set,
\begin{eqnarray*}
P_{(\mu,\Omega)}^n(1-\phi) &\leq& P_{(\mu,\Omega)}^n\Bigg(||\mu-\mu^*||-||\bar{X}-\mu||\leq\frac{M}{2}\epsilon\Bigg) \\
&\leq& P_{(\mu,\Omega)}^n\Bigg(||\bar{X}-\mu||>\frac{M\epsilon}{2}\Bigg) \\
&=& \mathbb{P}\Bigg(\frac{1}{\sqrt{n}}||\Sigma^{1/2}Z\Sigma^{1/2}||>\frac{M\epsilon}{2}\Bigg) \\
&\leq& \mathbb{P}\Big(||Z||^2\geq CM^2n\epsilon^2\Big).
\end{eqnarray*}
Finally, it is sufficient to bound $\mathbb{P}\Big(||Z||^2\geq CM^2n\epsilon^2\Big)$. We have
\begin{eqnarray*}
&& \mathbb{P}\Big(||Z||^2\geq CM^2n\epsilon^2\Big) \\
&=& \mathbb{P}\Bigg(\sum_{j=1}^p (Z_j^2-1)\geq CM^2n\epsilon^2-p\Bigg) \\
&\leq& \mathbb{P}\Bigg(\sum_{j=1}^p (Z_j^2-1)\geq CM^2n\epsilon^2/2\Bigg) \\
&\leq& \exp\Big(-C\min \big((M^2n\epsilon^2)^2/p, M^2n\epsilon^2\big)\Big) \\
&=& \exp\Big(-CM^2n\epsilon^2\Big),
\end{eqnarray*}
where we have used Bernstein's inequality. The proof is complete.
\end{proof}

\begin{lemma} \label{lem:testcov}
Assume $||\Sigma^*||\vee||\Omega^*||\leq\Lambda= O(1)$ and $||\Sigma_1||\vee||\Omega_1||\leq 2\Lambda$. There exist small $\delta,\delta',\bar{\delta}>0$ only depending on $\Lambda$ such that for any $M>0$, there exists  a testing function $\phi$ such that
$$P_{(\mu^*,\Omega^*)}^n\phi\leq 2\exp\Big(-C\delta'||\Sigma^*-\Sigma_1||_F^2\Big),$$
$$\sup_{\{(\mu,\Omega)\in\text{supp}(\Pi):||\Sigma-\Sigma_1||_F\leq\delta ||\Sigma^*-\Sigma_1||_F,||\mu-\mu^*||\leq M\epsilon\}}P_{\Sigma}^n(1-\phi)\leq 2\exp\Big(-C\delta'||\Sigma^*-\Sigma_1||_F^2\Big),$$
for some constant $C>0$, whenever $6\Lambda M^2\epsilon^2\leq \bar{\delta}||\Sigma_1-\Sigma^*||_F^2.$
\end{lemma}
\begin{proof}
Since the lemma is a slight variation of Lemma 5.9 in \cite{gao13}. We do not write the proof in full details. We choose to highlight the part where the current form is different from that in \cite{gao13}, and omit the similar part where the readers may find its full details in the proof of Lemma 5.9 in Gao and Zhou. We use the testing function
$$\phi=\left\{\frac{1}{n}\sum_{i=1}^n(X_i-\mu^*)^T(\Omega^*-\Omega_1)(X_i-\mu^*)>\log\det(\Omega\Sigma_1)\right\}.$$
We immediately have
$$P_{(\mu^*,\Omega^*)}^n\phi\leq 2\exp\Big(-C\delta'||\Sigma_1-\Sigma^*||_F^2\Big),$$
as is proved in \cite{gao13}. Now we are going to bound $P_{(\mu,\Omega)}^n(1-\phi)$ for every $(\mu,\Omega)$ in the alternative set. Note that we have
\begin{eqnarray*}
&& 1-\phi \\
&=& \Bigg\{\frac{1}{n}\sum_{i=1}^n(X_i-\mu)^T(\Omega^*-\Omega_1)(X_i-\mu)+2(\bar{X}-\mu)^T(\Omega^*-\Omega)(\mu-\mu^*)\\
&& +(\mu-\mu^*)^T(\Omega^*-\Omega_1)(\mu-\mu^*)<\log\det(\Omega\Sigma_1)\Bigg\}\\
&=& \Bigg\{\frac{1}{n}\sum_{i=1}^n\Bigg((X_i-\mu)^T(\Omega^*-\Omega_1)(X_i-\mu)-P_{(\mu,\Omega)}(X_i-\mu)^T(\Omega^*-\Omega_1)(X_i-\mu)\Bigg) \\
&& +2(\bar{X}-\mu)^T(\Omega^*-\Omega_1)(\mu-\mu^*)+(\mu-\mu^*)^T(\Omega_1-\Omega^*)(\mu-\mu^*)>\bar{\rho}\Bigg\},
\end{eqnarray*} 
where we have proved in \cite{gao13} that
$$\bar{\rho}\geq \bar{\delta}||\Sigma_1-\Sigma^*||_F^2,$$
for some $\bar{\delta}$ only depending on $\Lambda$. Using union bound, we have
\begin{eqnarray*}
&& P_{(\mu,\Omega)}^n(1-\phi) \\
&\leq& P_{(\mu,\Omega)}^n \Bigg\{\frac{1}{n}\sum_{i=1}^n\Bigg((X_i-\mu)^T(\Omega^*-\Omega_1)(X_i-\mu)-P_{(\mu,\Omega)}(X_i-\mu)^T(\Omega^*-\Omega_1)(X_i-\mu)\Bigg)>\frac{\bar{\rho}}{2}\Bigg\} \\
&& + P_{(\mu,\Omega)}^n\left\{2(\bar{X}-\mu)^T(\Omega^*-\Omega_1)(\mu-\mu^*)+(\mu-\mu^*)^T(\Omega_1-\Omega^*)(\mu-\mu^*)>\frac{\bar{\rho}}{2}\right\}.
\end{eqnarray*}
\cite{gao13} showed that the first term above is bounded by $2\exp\Big(-C\delta'||\Sigma_1-\Sigma^*||_F^2\Big)$. It is sufficient to bound the second term to close the proof. Actually, this is the only difference between this proof and the one in \cite{gao13}. Note that
$$|(\mu-\mu^*)^T(\Omega_1-\Omega^*)(\mu-\mu^*)|\leq 3\Lambda ||\mu-\mu^*||^2\leq 3\Lambda M^2\epsilon^2.$$
By assumption,
$$3\Lambda M^2\epsilon^2\leq \frac{1}{2}\bar{\delta}||\Sigma_1-\Sigma^*||_F^2\leq \frac{\bar{\rho}}{4}.$$
Hence,
\begin{eqnarray*}
&& P_{(\mu,\Omega)}^n\left\{2(\bar{X}-\mu)^T(\Omega^*-\Omega_1)(\mu-\mu^*)+(\mu-\mu^*)^T(\Omega_1-\Omega^*)(\mu-\mu^*)>\frac{\bar{\rho}}{2}\right\} \\
&\leq& P_{(\mu,\Omega)}^n\left\{2(\bar{X}-\mu)^T(\Omega^*-\Omega_1)(\mu-\mu^*)>\frac{\bar{\rho}}{4}\right\} \\
&=& \mathbb{P}\Bigg(Z^Ta>\frac{\bar{\rho}}{8}\Bigg),
\end{eqnarray*}
where $Z\sim N(0,1)$ and $a=\Sigma^{1/2}(\Omega^*-\Omega)(\mu-\mu^*)$. Using Hoeffding's inequality (see, for example, Proposition 5.10 of \cite{vershynin10}), we have
$$\mathbb{P}\Bigg(Z^Ta>\frac{\bar{\rho}}{8}\Bigg)\leq \exp\Big(-\frac{C\bar{\rho}^2}{||a||^2}\Big),$$
where
$$||a||^2\leq 9\Lambda^3||\mu-\mu^*||^2\leq 9\Lambda^3M^2\epsilon^2\leq \frac{3\Lambda^2}{4}\bar{\rho},$$
according to the assumption. Thus,
$$\mathbb{P}\Bigg(Z^Ta>\frac{\bar{\rho}}{8}\Bigg)\leq\exp\Big(-C\delta'||\Sigma_1-\Sigma^*||_F^2\Big),$$
for some $\delta'$ only depending on $\Lambda$. Therefore, $P_{(\mu,\Omega)}^n(1-\phi)\leq\exp\Big(-C\delta'||\Sigma_1-\Sigma^*||_F^2\Big)$ for all $(\mu,\Omega)$ in the alternative set and the proof is complete.
\end{proof}

\begin{proof}[Proof of Theorem \ref{thm:LDAgauss} and Theorem \ref{thm:QDAgauss}]
According to the remark after Theorem \ref{thm:QDA},
$$\Pi(A_{n}|X^n,Y^n)=\Pi_X(A_{X,n}|X^n)\Pi_Y(A_{Y,n}|Y^n).$$
Thus, it is sufficient to show both $\Pi_X(A_{X,n}|X^n)$ and $\Pi_Y(A_{Y,n}|Y^n)$ converge to $1$ in probability. Since they have the same form, we treat them together by omitting the subscript $X$ and $Y$. The posterior distribution is defined as
$$\Pi(A_n^c|X^n) = \frac{\int_{A_n^c}\exp\Big(l_n(\mu,\Omega)-l_n(\mu^*,\Omega^*)\Big)d\Pi(\mu,\Omega)}{\int\exp\Big(l_n(\mu,\Omega)-l_n(\mu^*,\Omega^*)\Big)d\Pi(\mu,\Omega)}=\frac{N_n}{D_n},$$
where we consider
$$A_n=\left\{||\mu-\mu^*||\leq M\sqrt{\frac{p^2\log n}{n}}, ||\Sigma-\Sigma^*||_F\leq \bar{M}\sqrt{\frac{p^2\log n}{n}}\right\},$$
for some $M$ and $\bar{M}$ sufficiently large. We are going to establish a test between the following hypotheses:
$$H_0: (\mu,\Omega)=(\mu^*,\Omega^*)\quad\text{vs}\quad H_1: (\mu,\Omega)\in A_n^c\cap\text{supp}(\Pi).$$
Decompose $A_n^c$ as
$$A_n^c=B_{1n}\cup B_{2n},$$
where
$$B_{1n}=\left\{||\mu-\mu^*||> M\sqrt{\frac{p^2\log n}{n}}\right\},$$
and
$$B_{2n}=\left\{||\mu-\mu^*||\leq M\sqrt{\frac{p^2\log n}{n}}, ||\Sigma-\Sigma^*||_F> \bar{M}\sqrt{\frac{p^2\log n}{n}}\right\}.$$
By Lemma \ref{lem:testmean}, there exists $\phi_1$ such that
$$P_{(\mu^*,\Omega^*)}^n\phi_1\vee \sup_{\text{supp}(\Pi)\cap B_{1n}}P_{(\mu,\Omega)}^n(1-\phi_1)\leq \exp\Big(-CM^2p^2\log n\Big).$$
For $B_{2n}$, we pick a covering set $\{\Sigma_j\}_{j=1}^N\subset B_{2n}\cap\text{supp}(\Pi)$, such that
$$B_{2n}\subset \cup_{j=1}^N B_{2nj},$$
where
$$B_{2nj}=\left\{||\mu-\mu^*||\leq M\sqrt{\frac{p^2\log n}{n}}, ||\Sigma-\Sigma_j||_F\leq \sqrt{\frac{p^2\log n}{n}}\right\},$$
and the covering number $N$ can be chosen to satisfy
$$\log N\leq Cp^2\log n,$$
as is shown in detail in the proof of Lemma \ref{lem:gaussprior}.
We may choose $\bar{M}$ large enough so that the assumption of Lemma \ref{lem:testcov} is satisfied, which implies the existence of $\phi_{2j}$ such that
$$P_{(\mu^*,\Omega^*)}^n\phi_{2j}\vee \sup_{\text{supp}(\Pi)\cap B_{2nj}}P_{(\mu,\Omega)}^n(1-\phi_{2j})\leq \exp\Big(-C\bar{M}^2p^2\log n\Big).$$
Define the final test as $\phi=\max\big(\phi_1,\vee_{j=1}^N \phi_{2j}\big)$. Then using union bound, we have
$$P_{(\mu^*,\Omega^*)}^n\phi\vee \sup_{\text{supp}(\Pi)\cap A_n^c}P_{(\mu,\Omega)}^n(1-\phi)\leq \exp\Big(-C(\bar{M}^2\wedge M^2)p^2\log n\Big),$$
for large $M$ and $\bar{M}$. Combining the testing result and the conclusions from Lemma \ref{lem:KLDA} and Lemma \ref{lem:priorDA},  we have
$$P_{(\mu^*,\Omega^*)}^n\Pi(A_n^c|X^n)=o_P(1),$$
by using the same argument in the proof of Lemma \ref{lem:gaussprior}. For QDA, as long as $p^2=o\Big(\frac{\sqrt{n}}{\log n}\Big)$, $A_n$ satisfies the requirement. For LDA, we use a $A_n$ defined as
$$A_n=\left\{||\mu_X-\mu_X^*||\vee||\mu_Y-\mu_Y^*||\leq M\sqrt{\frac{p^2\log n}{n}}, ||\Sigma-\Sigma^*||_F\leq \bar{M}\sqrt{\frac{p^2\log n}{n}}\right\}.$$
The proof needs some slight modification (including the previous lemmas) which is not essential and we choose to omit here. When $p^2=o\Big(\frac{\sqrt{n}}{\log n}\Big)$ is true, $A_n$ also satisfies the requirement there.

Now we are going to check the second conditions of Theorem \ref{thm:LDA} and Theorem \ref{thm:QDA}. We mainly sketch the QDA case. Using the notation in Lemma \ref{lem:gaussprior},
$$\int_{A_n}\exp\Big(l_n(\Omega_t,\mu_t)\Big)d\Pi(\Omega,\mu)=\xi_p^{-1}\int_{A_n}\exp\Big(l_n(\Omega_t,\mu_t)-\frac{1}{2}||\bar{\Omega}||_F^2-\frac{1}{2}||\mu||^2\Big)d\bar{\Omega}d\mu,$$
where $\xi_p$ is a normalizing constant and
\begin{eqnarray*}
&& \int_{A_n}\exp\Big(l_n(\Omega_t,\mu_t)-\frac{1}{2}||\bar{\Omega}||_F^2-\frac{1}{2}||\mu||^2\Big)d\bar{\Omega}d\mu \\
&=& \int_{A_n+(2tn^{-1/2}\Phi,tn^{-1/2}\xi)}\exp\Big(l_n(\Gamma,\theta)-\frac{1}{2}||\bar{\Gamma}-2tn^{1/2}\bar{\Phi}||_F^2-\frac{1}{2}||\theta-tn^{-1/2}\xi||^2\Big)d\bar{\Gamma}d\theta.
\end{eqnarray*}
Proceeding as in Lemma \ref{lem:gaussprior}, the result is proved.
\end{proof}

\section{Proof of Lemma \ref{lem:logdet}}

Let $\hat{\Omega}=\hat{\Sigma}^{-1}$.
Note that
\begin{eqnarray*}
&& \left|\log\det(\Sigma)-\log\det(\hat{\Sigma})-\text{tr}\Big((\Sigma-\hat{\Sigma})\Omega^*\Big)\right| \\
&\leq& \left|\log\det\Big(\hat{\Omega}^{1/2}\Sigma\hat{\Omega}^{1/2}-I+I\Big)-\text{tr}\Big(\hat{\Omega}^{1/2}\Sigma\hat{\Omega}^{1/2}-I\Big)\right| \\
&& + \left|\text{tr}\Big((\Sigma-\hat{\Sigma})(\Omega^*-\hat{\Omega})\Big)\right|.
\end{eqnarray*}
For the second term,
\begin{eqnarray*}
&& \sqrt{\frac{n}{p}}\left|\text{tr}\Big((\Sigma-\hat{\Sigma})(\Omega^*-\hat{\Omega})\Big)\right| \\
&\leq& C\sqrt{\frac{n}{p}}||\Sigma-\hat{\Sigma}||_F||\hat{\Sigma}-\Sigma^*||_F \\
&\leq& C\sqrt{\frac{n}{p}}\Big(||\hat{\Sigma}-\Sigma^*||_F^2+||\Sigma-\Sigma^*||_F||\hat{\Sigma}-\Sigma^*||_F\Big) \\
&\leq& O_P\Bigg(\sqrt{\frac{p^3}{n}}+\sqrt{p}||\Sigma-\Sigma^*||_F\Bigg),
\end{eqnarray*}
which converges to zero whenever $\sqrt{p}||\Sigma-\Sigma^*||_F\leq \delta_n=o(1)$. For the first term,
\begin{eqnarray*}
&& \sqrt{\frac{n}{p}}\left|\log\det\Big(\hat{\Omega}^{1/2}\Sigma\hat{\Omega}^{1/2}-I+I\Big)-\text{tr}\Big(\hat{\Omega}^{1/2}\Sigma\hat{\Omega}^{1/2}-I\Big)\right| \\
&\leq& C\sqrt{\frac{n}{p}}\left\|\hat{\Omega}^{1/2}\Sigma\hat{\Omega}^{1/2}-I\right\|_F^2 \\
&\leq& C\sqrt{\frac{n}{p}}\Big(||\Sigma-\Sigma^*||_F^2+||\hat{\Sigma}-\Sigma^*||_F^2\Big) \\
&\leq& O_P\Bigg(\sqrt{\frac{p^3}{n}}+\sqrt{\frac{n}{p}}||\Sigma-\Sigma^*||_F^2\Bigg),
\end{eqnarray*}
which converges to zero whenver $\sqrt{\frac{n}{p}}||\Sigma-\Sigma^*||_F^2\leq\delta_n=o(1)$. Thus, the proof is complete.

\section{Proof of Lemma \ref{lem:coveigen}, Lemma \ref{lem:precisioneigen} \& Proposition \ref{prop:eigencounter}}

Due to the similarity between Lemma \ref{lem:coveigen} and Lemma \ref{lem:precisioneigen},
we only give the  proof of Lemma \ref{lem:coveigen}.
Let us study the linear approximation of eigenvalue perturbation. In particular, we are going to find the first-order Taylor expansion of $\lambda_m(\Sigma)-\lambda_m(\hat{\Sigma})$ and control the error term in some set $A_n$. We have the following spectral decomposition for the three covariance matrices $\Sigma,\hat{\Sigma},\Sigma^*$.
$$\Sigma=UDU^T,\quad \hat{\Sigma}=\hat{U}\hat{D}\hat{U}^T,\quad \Sigma^*=U^*D^*U^{*T}.$$
Denote the $m$-th column of $U,\hat{U},U^*$ by $u_m,\hat{u}_m,u_m^*$.
Then,
\begin{eqnarray*}
\lambda_m(\Sigma)-\lambda_m(\hat{\Sigma}) &=& \lambda_m(\hat{U}^T\Sigma\hat{U})-\lambda_m(\hat{U}^T\hat{\Sigma}\hat{U}) \\
&=& \lambda_m(\hat{D}+\hat{U}^T(\hat{\Sigma}-\Sigma)\hat{U})-\lambda_m(\hat{D}).
\end{eqnarray*}
Write $A=\hat{D}$ and $\Delta=\hat{U}^T(\hat{\Sigma}-\Sigma)\hat{U}$. The problem is reduced to the eigenvalue perturbation of a diagonal matrix. According to the expansion formula in \cite{kato95} and \cite{ames12}, we have
\begin{equation}
\lambda_m(A+\Delta)-\lambda_m(A)=\lambda_m'(A,\Delta)+\sum_{k=2}^{\infty}\lambda_m^{(k)}(A,\Delta), \label{eq:kato}
\end{equation}
where the first-order term is
$$\lambda_m'(A,\Delta)=\Delta_{mm}=\text{tr}\Big(\hat{U}^T(\hat{\Sigma}-\Sigma)\hat{U} E_{mm}\Big)=\text{tr}\Big((\hat{\Sigma}-\Sigma)\hat{U}E_{mm}\hat{U}^T\Big)=\text{tr}\Big((\hat{\Sigma}-\Sigma)\hat{u}_m\hat{u}_m^T\Big).$$
In the remainder term, we have
$$\lambda_m^{(k+1)}(A,\Delta)=-\frac{1}{k+1}\sum_{v_1+...+v_{k+1}=k,v_1,...,v_{k+1}\geq 0}\text{tr}\Big(\Delta \tilde{A}^{v_1}...\Delta \tilde{A}^{v_{k+1}}\Big),$$
where $\tilde{A}^{v}$ is the matrix power when $v\geq 1$ with an exception that $\tilde{A}^0=-e_m e_m^T$, where $e_m$ is the $m$-th vector of the canonical basis of $\mathbb{R}^p$. The matrix $\tilde{A}$ is defined as
$$\tilde{A}=\sum_{1\leq j\leq p, j\neq m}\frac{e_je_j^T}{a_m-a_j},$$
where $a_j=\lambda_j(\hat{\Sigma})$ is the $(j,j)$-th entry of $A$.
Therefore, for any integer $v\geq 1$,
$$||\tilde{A}^{v}||=||\tilde{A}||^{v}\leq \max\left\{\frac{1}{|a_m-a_{m-1}|},\frac{1}{|a_m-a_{m+1}|}\right\}.$$
We are going to show that the first term in (\ref{eq:kato}) is a good enough approximation of $\lambda_m(A+\Delta)-\lambda_m(A)$ by bounding the higher-order terms.
Let us provide a bound for $|\lambda_m^{(k+1)}(A,\Delta)|$. Let $\mathbb{N}=\{0,1,2,...\}$. Consider the set
$$\left\{(v_1,...,v_{k+1})\in\mathbb{N}: v_1+...+v_{k+1}=k\right\}.$$
From its definition, there must be some $l$, such that $v_l=0$ to satisfy $v_1+...+v_{k+1}=k$. Thus, the set can be decomposed into a union of disjoint subsets as follows,
\begin{eqnarray*}
&& \left\{(v_1,...,v_{k+1})\in\mathbb{N}^{k+1}: v_1+...+v_{k+1}=k\right\} \\
&=& \bigcup_{l=1}^{k+1}\left\{(v_1,...,v_{l-1},v_{l+1},...,v_{k+1})\in\mathbb{N}^{k}: v_1+...+v_{l-1}+v_{l+1}+...+v_{k+1}=k\right\}.
\end{eqnarray*}
Clearly, each the cardinality of each subset is
$${2k-1\choose k-1}\leq (3e)^k.$$
We give names to the sets we have mentioned by
$$V_{k+1}=\cup_{l=1}^{k+1}V_{k+1,l}.$$
For $l=1$, we have
\begin{eqnarray*}
\left|\sum_{V_{k+1,1}}\text{tr}\Big(\Delta \tilde{A}^{v_1}...\Delta \tilde{A}^{v_{k+1}}\Big)\right| &\leq& \sum_{V_{k+1},1}\left|\text{tr}\Big(\Delta \tilde{A}^{v_1}...\Delta \tilde{A}^{v_{k+1}}\Big)\right| \\
&=& \sum_{V_{k+1},1}\left|\text{tr}\Big(\Delta \tilde{A}^{0}\Delta \tilde{A}^{v_2}...\Delta \tilde{A}^{v_{k+1}}\Big)\right| \\
&=& \sum_{V_{k+1},1}\left|\text{tr}\Big(\Delta e_me_m^T\Delta \tilde{A}^{v_2}...\Delta \tilde{A}^{v_{k+1}}\Big)\right| \\
&\leq& \sum_{V_{k+1},1}\left\|\Delta e_m\right\|\left\|e_m^T\Delta \tilde{A}^{v_2}...\Delta \tilde{A}^{v_{k+1}}\right\| \\
&\leq& \sum_{V_{k+1},1}||\Delta||^{k+1} ||\tilde{A}||^{v_2+...+v_{k+1}} \\
&=& \sum_{V_{k+1},1}||\Delta||^{k+1} ||\tilde{A}||^{k} \\
&=& ||\Delta|| \Big(3e ||\Delta||||\tilde{A}||\Big)^k 
\end{eqnarray*}
In the same way, the bound also holds for other $l$. Therefore,
\begin{eqnarray*}
|\lambda_m^{(k+1)}(A,\Delta)| &=& \left|\frac{1}{k+1}\sum_{l=1}^{k+1} \sum_{V_{k+1,l}}\text{tr}\Big(\Delta \tilde{A}^{v_1}...\Delta \tilde{A}^{v_{k+1}}\Big)\right| \\
&\leq& \frac{1}{k+1}\sum_{l=1}^{k+1}||\Delta|| \Big(3e ||\Delta||||\tilde{A}||\Big)^k \\
&=& ||\Delta|| \Big(3e ||\Delta||||\tilde{A}||\Big)^k.
\end{eqnarray*}
When $3e ||\Delta||||\tilde{A}||<1$, we may sum over $k$, and obtain
\begin{eqnarray*}
\left|\sum_{k=2}^{\infty}\lambda_m^{(k)}(A,\Delta)\right| &\leq& ||\Delta||\sum_{k=1}^{\infty}\Big(3e ||\Delta||||\tilde{A}||\Big)^k \\
&\leq& \frac{3e||\Delta||^2||\tilde{A}||}{1-3e||\Delta||||\tilde{A}||}.
\end{eqnarray*}
Note that
$$||\Delta||=||\hat{\Sigma}-\Sigma||\leq ||\hat{\Sigma}-\Sigma^*||+||\Sigma-\Sigma^*||\leq O_{P}\Big(\sqrt{\frac{p}{n}}\Big)+||\Sigma-\Sigma^*||,$$
and
$$||\tilde{A}||\leq C\min\left\{\max\{|\lambda_m-\lambda_{m-1}|^{-1}, |\lambda_m-\lambda_{m+1}|^{-1}\}, ||\hat{\Sigma}-\Sigma||^{-1}\right\}= O_P\Big(\min\big(\delta^{-1},\sqrt{n/p}\big)\Big).$$
Therefore,
$$3e||\Delta||||\tilde{A}||= O_P\Bigg(\frac{\sqrt{p/n}+||\Sigma-\Sigma^*||}{\delta}\Bigg)= o_P(1),$$
holds under the assumption $\delta^{-1}\sqrt{p/n}=o(1)$ and when $\delta^{-1}||\Sigma-\Sigma^*||=o_P(1)$. The remainders are controlled by
$$\sqrt{n}\left|\sum_{k=2}^{\infty}\lambda_m^{(k)}(A,\Delta)\right|= O_P\Bigg(\frac{p}{\delta \sqrt{n}}+\frac{\sqrt{n}||\Sigma-\Sigma^*||^2}{\delta}\Bigg)= o_P(1),$$
under the assumption $\frac{p}{\delta\sqrt{n}}= o(1)$ and when $\delta^{-1}\sqrt{n}||\Sigma-\Sigma^*||^2=o(1)$. Hence, by (\ref{eq:kato}), we have proved
$$\sup_{\{\delta^{-1}||\Sigma-\Sigma^*||\vee\delta^{-1}\sqrt{n}||\Sigma-\Sigma^*||^2\leq\delta_n\}}\sqrt{n}\left|\lambda_m(A+\Delta)-\lambda_m(A)-\lambda_m'(A,\Delta)\right|= o_P(1),$$
for any $\delta_n= o(1)$.

Finally, for the first order term $\lambda_m'(A,\Delta)$, we approximate it by $\text{tr}\Big((\hat{\Sigma}-\Sigma)u_m^*u_{m}^{*T}\Big)$, and the approximation error is
$$\left|\text{tr}\Big((\hat{\Sigma}-\Sigma)(\hat{u}_m\hat{u}_m^T-u_m^*u_m^{*T})\Big)\right|=||u_m^*u_m^{*T}-\hat{u}_m\hat{u}_m^T||_F \left|\text{tr}\Big((\hat{\Sigma}-\Sigma)K\Big)\right|,$$
where $K$ is a rank-two unit Frobenius norm matrix. It has SVD $K=c_1d_1d_1^T+c_2d_2d_2^T$, with $c_1\vee c_2\leq 1$. Therefore,
$$||u_m^*u_m^{*T}-\hat{u}_m\hat{u}_m^T||_F \left|\text{tr}\Big((\hat{\Sigma}-\Sigma)K\Big)\right|\leq C||\hat{\Sigma}-\Sigma|| ||\hat{\Sigma}-\Sigma^*||\leq O_{P}\Big(\frac{p}{n}\Big)+O_{P}\Big(\sqrt{\frac{p}{n}}||\Sigma-\Sigma^*||\Big).$$
Under the assumption $p= o(n)$, when $\sqrt{p}||\Sigma-\Sigma^*||=o(1)$, we have
$$\sqrt{n}\left|\text{tr}\Big((\hat{\Sigma}-\Sigma)(\hat{u}_m\hat{u}_m^T-u_m^*u_m^{*T})\Big)\right|= o_P(1).$$
Therefore, the proof of Lemma \ref{lem:coveigen} is complete.

Now we prove Proposition \ref{prop:eigencounter}. We redefine $A=D^*$ and  $\Delta=U^{*T}(\Sigma^*-\hat{\Sigma})U^*$ and correspondingly $\tilde{A}^v$. In the case where $\delta$ is a constant, we have
$$||\tilde{A}||\leq C,\quad ||\Delta||= O_P\Bigg(\sqrt{\frac{p}{n}}\Bigg).$$
Similar to (\ref{eq:kato}), we have
$$\lambda_1(\hat{\Sigma})-\lambda_1(\Sigma^*)=\lambda_1'(A,\Delta)+\sum_{k=2}^{\infty}\lambda_1^{(k)}(A,\Delta),$$
where $\lambda_1'(A,\Delta)=\text{tr}\Big((\Sigma^*-\hat{\Sigma})u_1^*u_1^{*T}\Big)$ and thus $\sqrt{n}\lambda_1'(A,\Delta)$ is asymptotically normal. For the remainder term, we decompose it as
$$
\sum_{k=2}^{\infty}\lambda_1^{(k)}(A,\Delta) = \lambda_1^{(2)}(A,\Delta)+\sum_{k=3}^{\infty}\lambda_1^{(k)}(A,\Delta).
$$
Using similar techniques in proving Lemma \ref{lem:coveigen}, we have
\begin{eqnarray*}
\sqrt{n}\left|\sum_{k=3}^{\infty}\lambda_1^{(k)}(A,\Delta)\right| &\leq& \sqrt{n}\sum_{k=2}^{\infty}||\Delta||\Big(3e||\Delta||||\tilde{A}||\Big)^k \\
&\leq& C\sqrt{n}||\Delta||^3 \\
&=& O_P\Bigg(\sqrt{\frac{p^3}{n^2}}\Bigg),
\end{eqnarray*}
which is $o_P(1)$ under the assumption. Therefore,
$$\sqrt{n}\Big(\lambda_1(\hat{\Sigma})-\lambda_1(\Sigma^*)\Big)=\sqrt{n}\text{tr}\Big((\Sigma^*-\hat{\Sigma})u_1^*u_1^{*T}\Big)+\sqrt{n}\lambda_1^{(2)}(A,\Delta)+o_P(1).$$
It remains to show that $\sqrt{n}\lambda_1^{(2)}(A,\Delta)$ is not $o_P(1)$. Note that $\tilde{A}=\sum_{j\geq 2}\frac{e_je_j^T}{a_1-a_j}$, with $a_j=\lambda_j(\Sigma^*)$. Hence,
\begin{eqnarray*}
\sqrt{n}\left|\lambda_1^{(2)}(A,\Delta)\right| &=& \sqrt{n}\left|\text{tr}\Big(\Delta\tilde{A}\Delta e_1e_1^T\Big)\right| \\
&=& \sqrt{n}\left|\sum_{j\geq 2}\frac{1}{a_1-a_j}|e_1^TU^{*T}(\Sigma^*-\hat{\Sigma})U^*e_j|^2\right| \\
&\geq& \frac{\sqrt{n}}{a_1-a_2}\sum_{j\geq 2}|e_1^TU^{*T}(\Sigma^*-\hat{\Sigma})U^*e_j|^2.
\end{eqnarray*}
For fixed eigengap, $a_1-a_2$ is a constant.
When $\Sigma^*$ is diagonal, $U^*=I$, and we have $|e_1^TU^{*T}(\Sigma^*-\hat{\Sigma})U^*e_j|^2=\hat{\sigma}_{1j}^2$. Moreover, the fact that $\Sigma^*$ is diagonal implies $\hat{\sigma}_{1j}^2$ are independent for $j=2,...,p$.  Hence,
$$\sqrt{n}\left|\lambda_1^{(2)}(A,\Delta)\right| \geq C\sqrt{n}\sum_{j=2}^p\hat{\sigma}_{1j}^2,$$
where $\sqrt{n}\sum_{j=2}^p\hat{\sigma}_{1j}^2$ is at the level of $p/\sqrt{n}$, which diverges to $\infty$ under the assumption. The proof is complete.

\bibliographystyle{biometrika}
\bibliography{reference}


\end{document}